\documentclass[11pt]{amsart}

\usepackage{amsfonts,amsmath,latexsym,amssymb,verbatim,amsbsy,amsthm}
\usepackage{esint}

\usepackage{graphicx}

\usepackage[top=1in, bottom=1in, left=1in, right=1in]{geometry}

\usepackage[dvipsnames]{xcolor}

\usepackage[colorlinks=true, pdfstartview=FitV, linkcolor=RoyalBlue,citecolor=ForestGreen, urlcolor=blue]{hyperref}

\theoremstyle{plain}
\newtheorem{THEOREM}{Theorem}[section]

\newtheorem{theorem}[THEOREM]{Theorem}

\newtheorem{lemma}[THEOREM]{Lemma}
\newtheorem{proposition}[THEOREM]{Proposition}

\theoremstyle{definition}

\theoremstyle{remark}

\newtheorem{remark}[THEOREM]{Remark}

\newcommand{\thm}[1]{Theorem~\ref{#1}}
\newcommand{\lem}[1]{Lemma~\ref{#1}}

\newcommand{\prop}[1]{Proposition~\ref{#1}}

\newcommand{\sect}[1]{Section~\ref{#1}}

\def \a {\alpha}
\def \b {\beta}

\def \d {\delta}

\def \e {\varepsilon}

\def \n {\nabla}

\def \t {\tau}
\def \th {\theta}

\def \D {\Delta}

\def \L {\Lambda}

\def \O {\Omega}

\def \be {{\bf e}}

\def \bj {{\bf j}}
\def \bk {{\bf k}}

\def \cC {\mathcal{C}}

\def \cL {\mathcal{L}}

\def \dH {\dot{H}}\def \dW {\dot{W}}
\def \rmin{\underline{\rho}}
\def \rmax{\overline{\rho}}

\newcommand{\N}{\ensuremath{\mathbb{N}}}   
\newcommand{\Z}{\ensuremath{\mathbb{Z}}}   
\newcommand{\T}{\ensuremath{\mathbb{T}}}   

\def \lan {\langle}
\def \ran {\rangle}

\def \p {\partial}

\def \ss {\subset}

\def \dx  {\, \mbox{d}x}
\def \dxi  {\, \mbox{d}\xi}

\def \dy  {\, \mbox{d}y}
\def \dz  {\, \mbox{d}z}

\def \ds  {\, \mbox{d}s}

\def \dth  {\, \mbox{d}\th}

\def \ddt  {\frac{\mbox{d\,\,}}{\mbox{d}t}}

\begin{document}

\title[Well-posedness of topological models]{Local well-posedness of the topological Euler alignment models of collective behavior}

\author{David N. Reynolds}

\author{Roman Shvydkoy}
\address{Department of Mathematics, Statistics and Computer Science, University of Illinois at Chicago, 60607}

\email{dreyno8@uic.edu}
\email{shvydkoy@uic.edu}


\date{\today}

\thanks{\textbf{Acknowledgment.}  
	This work was  supported in part by NSF
	grant DMS-1813351.}

\begin{abstract}
In this paper we address the problem of well-posedness of multi-dimensional topological Euler-alignment models introduced in \cite{ST-topo}. The  main result demonstrates local existence and uniqueness of classical solutions in class $(\rho,u) \in H^{m+\a} \times H^{m+1}$ on the periodic domain $\T^n$, where $0<\a<2$ is the order of singularity of the topological communication kernel $\phi(x,y)$, and $m = m(n,\a)$ is large. Our approach is based on new sharp coercivity estimates for the topological alignment operator 
\[
\cL_\phi f(x) = \int_{\T^n} \phi(x,y) (f(y) - f(x) ) \dy,
\]
which render proper a priori estimates and help stabilize viscous approximation of the system. In dimension 1, this result, in conjunction with the technique developed in \cite{ST-topo} gives global well-posendess in the natural space of data mentioned above. 
\end{abstract}

\maketitle

\tableofcontents

\section{Introduction}

Several recent field studies on animal behavior revealed that in some cases communication 
between species is regulated by topological distance metric, which depends on the number of other species in close proximity rather than their Euclidean distance, see \cite{Bal2008,Wa2017} and references therein. Kinetic models 
interpreting such topological interactions as the $K$-nearest neighbor rule were studied at length by Blanchet and Degond in \cite{BD2016,BD2017}. In \cite{Ha2013} Haskovec defines topological asymmetric ``distance" between agents $x$ and $y$ by counting all agents in the ball of radius $|x-y|$ centered at $x$. It is shown that the classical Cucker-Smale model \cite{CS2007a,CS2007b} with kernel depending on such distance aligns under a global in time graph connectivity assumption -- one that is guaranteed to hold, for instance, for metric models with long range interactions given by $\phi(r) = \frac{H}{(1+r^2)^{\b/2}}$, $\b\leq 1$. Establishing flocking behavior under a strictly local communication rule, however, continues to be a challenging mathematical problem in the theory of collective motion.

\begin{figure}
	\includegraphics[width=5in]{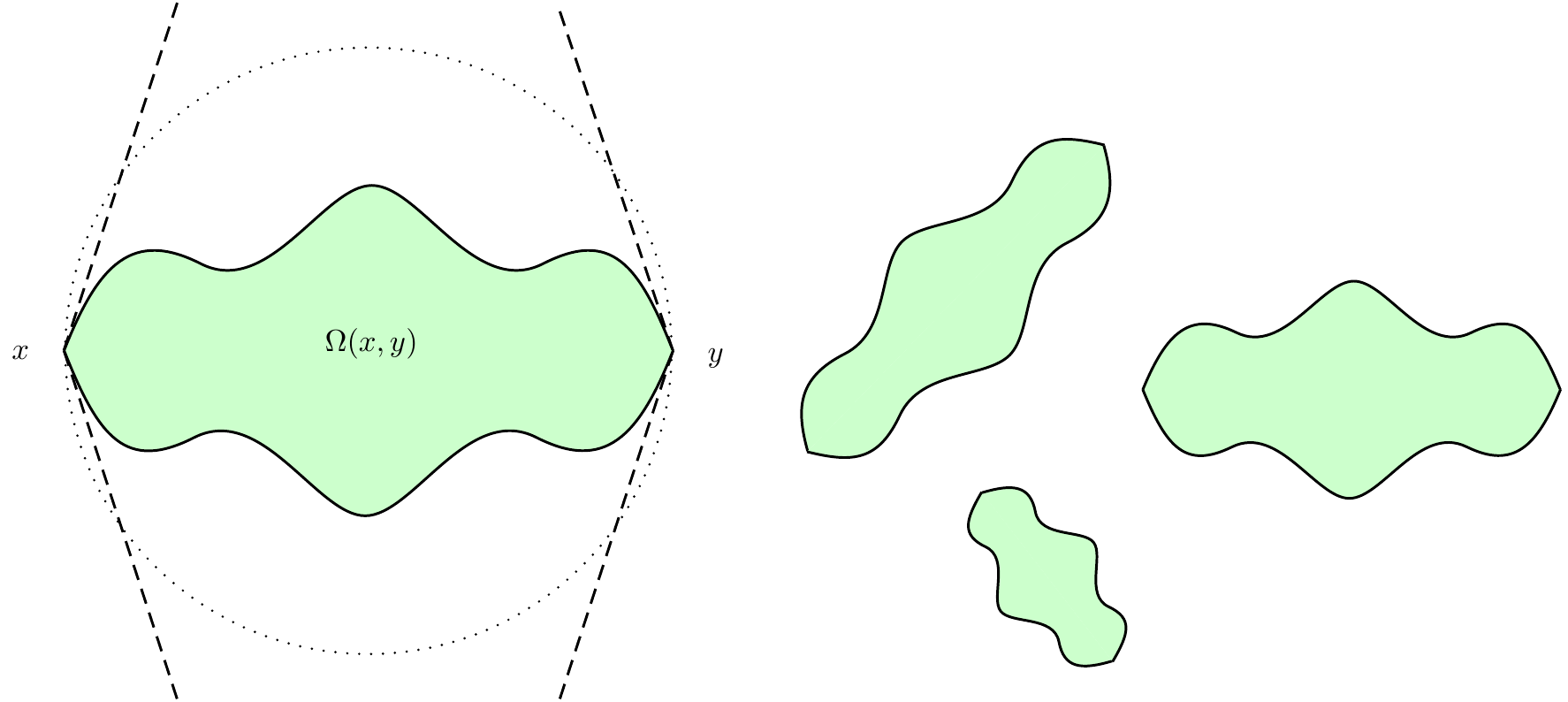}
	\caption{Communication domain satisfying assumptions (D1)--(D3)}\label{f:domain}
\end{figure}

In the context of hydrodynamic Cucker-Smale model, also known as the Euler alignment system, given by 
\begin{equation}\label{e:main}
\left\{
\begin{split}
\rho_t + \n \cdot (\rho u) & = 0, \\
u_t + u \cdot \n u &= \int_{\T^n}\phi(x,y)(u(t,y) - u(t,x))  \rho(t,y)\dy,
\end{split} \right. 
\end{equation}
a new \emph{local} and \emph{symmetric} kernel was introduced in \cite{ST-topo} with a 
mix of topological and metric components. Specifically, it is postulated that the communication strength between agents $(x,y)$ is inversely proportional to the mass of a symmetric region $\O(x,y) = \O(y,x)$  at time $t$ which is encoded into the topological quasi-distance function
\[
d(x,y) = \left(\int_{\O(x,y)} \rho(\xi,t) \dxi \right)^{1/n}.
\]
We define $\phi(x,y)$ as a non-convolution type singular kernel of degree $0<\a<2$ by
\begin{equation}\label{eq:kernel}
\phi(x,y)=\frac{h(x-y)}{|x-y|^{n+\alpha-\tau}d^{\tau}(x,y)},
\end{equation}
where $h = h(r)$ is a radial smooth bump function supported on a ball of radius $r_0$ -- a  communication cutoff scale, and $\t>0$ is a parameter that gauges  presence of topological effects in the system. Although the communication domain considered in \cite{ST-topo} is a specific football shaped body of revolution the results extend to any family of domains obtained by scaling of the basic domain $\O_0 = \O(-\be_1,\be_1)$ such that
\begin{itemize}
	\item[(D1)] $\p \O_0$ is smooth except at $\pm \be_1$ where it is Lipschitz of conical opening of degree $<\pi$,
	\item[(D2)] $\O_0 = - \O_0$,
	\item[(D3)] $\O_0 \ss B_1(0)$.
\end{itemize} 
Figure~\ref{f:domain} shows  example of a typical domain. If the topological component of the communication is sufficiently strong, then  all classical non-vacuous solutions to \eqref{e:main} align.
\begin{theorem}[\cite{ST-topo}]
Suppose $\t \geq n$. Then any classical solution $(u,\rho)$ to \eqref{e:main} on the torus $\T^n$ satisfying the hydrodynamic connectivity condition
\begin{equation}\label{e:connect}
\rho(x,t) \gtrsim \frac{1}{1+t}
\end{equation}
aligns to its conserved momentum $\bar{u}$ at a logarithmic rate
\[
\|u(t) - \bar{u} \|_\infty \lesssim \frac{1}{\sqrt{\ln t}}.
\]
\end{theorem} 
It is also shown that in 1D the condition \eqref{e:connect} holds automatically for all time, so the model exhibits unconditional alignment in this case. 

Regularity theory of metric models \eqref{e:main}, i.e. where $\phi(x,y) = \phi(|x-y|)$ has been developed extensively in  \cite{CCTT2016,CCMP2017,DMPW2019,HeT2017,DKRT2018,ST1,ST2,ST3,Shv2018,Tan2017,TT2014}, and is most completely understood only in one dimensional settings due to an extra conserved quantity
\begin{equation}
	e = u_x + \phi \ast \rho, \quad e_t + (u e)_x = 0,
\end{equation}
which allows to directly control the slope of $u$. For the smooth kernel case this leads to Burgers' type threshold condition $e_0 \geq 0$ to guarantee global existence. For singular communication, $\phi(r) = \frac{1}{r^{1+\a}}$, additional parabolic regularization leads to global existence and flocking for any smooth non-vacuous data on $\T$,  \cite{DKRT2018,ST1,ST2,ST3}. In multi-D, small initial data results were proved in \cite{DMPW2019,Shv2018,HeT2017}.

Topological models presented a new set of challenges from the perspective of regularity theory as they do not fit directly under any studied class of fractional drift diffusion equations for which H\"older regularization has been established, see \cite{SS2016,S2012}. The one dimensional case has been treated in the same article \cite{ST-topo} where global wellposedness in class $u\in H^{m+1}$, $\rho\in H^{m+\a/2}$ was proved for $\t\leq \a$. In dimension 1 the topological model shares  a similar conservation law with the metric one, given by 
\[
e = u_x + \cL_\phi \rho,
\]
where $\cL_\phi$ is the singular alignment operator associated with the topological kernel $\phi$:
\begin{equation}
	\cL_{\phi}f =\int_{\mathbb{T}^n} \phi(x,y)(f(y)-f(x)) \dy.
\end{equation}

The primary goal of this paper is to initiate the study of topological models in arbitrary dimension by establishing local well-posedness of classical solutions in high regularity Sobolev classes.
\begin{theorem}\label{t:main}
	For any initial data $u_0\in H^{m+1}(\T^n)$, $\rho_0 \in H^{m+\a}(\T^n)$,  $m\geq m(\a,n)$, with no vacuum $\rho_0(x) >0$ there exists a unique solution to the system \eqref{e:main}-\eqref{eq:kernel}-$(D2)$ on  a time interval $[0,T_0)$ with $T_0$ dependent on the initial condition,  in the class
	\begin{equation}\label{e:class}
		\begin{split}
		u & \in C_w([0,T_0), H^{m+1}) \cap L^2([0,T_0), H^{m+1+\frac{\a}{2}}) \\
		\rho& \in C_w([0,T_0), H^{m+\a})
		\end{split}
	\end{equation}
If $n=1$, then the solution is global and \eqref{e:class} holds on any finite time interval. 
\end{theorem}
Symbol $C_w$ here means weakly continuous functions. Let us make several remarks.  First,  the relationship between regularity classes of $u$ and $\rho$ are related naturally by the way they enter into the $e$-quantity already in 1D. Second, the global existence in dimension 1 is an improvement over \cite{ST-topo} in the density class which is achieved by establishing sharp coercivity estimates on the alignment operator (see \prop{p:maincoer}):
\begin{equation}\label{e:introcoerc}
	\|	\cL_{\phi}f \|_{\dH^m} \sim \| f\|_{\dH^{m+\a}} + \mathrm{lower\ order\ terms}.
\end{equation}
This is one of the major technical components in the proof of \thm{t:main}.  Third, we cannot assert that the density gains any additional $L^2$ integrability in higher class similar to the velocity.  The reason is that the continuity equation has no intrinsic parabolic structure as it does in 1D. Indeed, considering $e = \n \cdot u + \cL_\phi \rho$  in multi-D, it does not satisfy the clean continuity law, see \eqref{e:iden}, and consequently cannot be considered as a lower order quantity as in 1D. Writing the continuity equation as
\begin{equation}\label{e:cont2}
	\rho_t + u \cdot \n \rho + e\rho = \rho \cL_\phi \rho,
\end{equation}
injects a rough forcing term $e\rho$ that drives the density out of the expected smoother class $H^{m+\a+\frac{\a}{2}}$.

The local existence proof is based on establishing short term control on the grand quantity
\begin{equation}\label{e:Ym}
Y_m = \| u\|_{\dH^{m+1}}^2 + \|e\|_{\dH^m}^2 +  \|\rho\|_{\dH^m}^2 + \rmax + \rmin^{-1},
\end{equation}
where $\rmin = \min \rho$, $\rmax = \max \rho$. 
The overall goal is to establish an a priori Riccati type equation
\begin{equation}
\ddt Y_m \leq C Y_m^N,
\end{equation}
where $N\in \N$ may be large. Coercivity estimates \eqref{e:introcoerc} demonstrate that $Y_m$ is equivalent to controlling $u$ in $H^{m+1}$ and $\rho$ in $H^{m+\a}$. However, due to the deficiencies associated with the density equation \eqref{e:cont2}, we resort to replacing the pair $(u,\rho)$ with $(u,e)$ for the purposes of a priori estimates. The same strategy already appeared in all previous works on singular models \cite{DKRT2018,ST1,ST2,ST3,ST-topo}.

The structure of the paper is straightforward. In \sect{s:prelim} we set the notation and make elementary a priori estimates on lower order terms in $Y_m$. \sect{s:coerc} is entirely devoted to coercivity bounds on the alignment operator via commutator estimates. Sections~\ref{s:u} and \ref{s:e} detail a priori estimates on the $u$ and $e$ equations, respectively. In \sect{s:visc} we conclude by finding local solutions via viscous regularization scheme and establish stability of our a priori estimates under such approximation.

\section{Preliminaries}\label{s:prelim}

In this section we go through a few quick computations that establish a priori estimates on the lower order terms in the grand quantity $Y_m$ \eqref{e:Ym}, namely, $ \|\rho\|_{\dH^m}^2 + \rmax + \rmin^{-1}$.

The bound on  $\|\rho\|_{\dH^m}^2$ follows by a simple classical commutator estimate. Indeed, we have
\[
\rho_t + u \cdot \n \rho + (\n\cdot u) \rho =0.
\]
So, testing with $\p^{2m}\rho$ we obtain
\[ 
\ddt \|\rho\|_{\dH^m}^2 = \int (\n\cdot u) |\p^m \rho|^2 \dx - \int( \p^m(u \cdot \n \rho) - u \cdot \n \p^m \rho) \p^m \rho \dx - \int  \p^m((\n\cdot u) \rho) \p^m \rho \dx.
\]
Recalling the classical commutator estimate
\begin{equation}\label{e:classcomm}
\| \p^{m}(fg) - f \p^{m} g \|_2 \leq |\n f|_\infty \|g\|_{\dH^{m-1}} + \| f\|_{\dH^{m}} |g|_\infty,
\end{equation}
we obtain
\[
\ddt \|\rho\|_{\dH^m}^2 \leq  |\n u|_\infty  \|\rho\|_{\dH^m}^2 +  \|u\|_{\dH^m}  \|\rho\|_{\dH^m} |\n \rho|_\infty +  \|u\|_{\dH^{m+1}}  \|\rho\|_{\dH^m} |\rho|_\infty \leq C Y_m^3.
\]
Next, differentiating the maximum we obtain
\[
\ddt \rmax \leq |\n u|_\infty \rmax,
\]
and similarly,
\[
\ddt \rmin^{-1} \leq |\n u|_\infty \rmin^{-1}.
\]
Thus,
\[
\ddt (\|\rho\|_{\dH^m}^2 + \rmax +  \rmin^{-1}) \lesssim Y_m^3.
\]
Having these simple bounds out of the way, the main focus now will be on obtaining similar bounds on the first two components of $Y_m$ and ensuring that $Y_m$ is comparable with the spaces in which we are proving local well-posedness.

\section{Coercivity bounds on $\cL_\phi$}\label{s:coerc}

Letting $y=x+z$ and defining the increment operator $\delta_zf(x)=f(x+z)-f(x)$ we can rewrite the operator as
\begin{align}
\mathcal{L}_{\phi}f&=\int_{\mathbb{T}^n} \phi(x,x+z)\delta_zf(x) \dz
\end{align}

\begin{proposition} \label{p:maincoer}For any sufficiently large $m\in\N$ and  $0<\a<2$  there exists a polynomial $p_N$ of degree $N = N(m,n,\a) \in \N$  such that the following inequalities hold
	\begin{equation}\label{e:Lcommmain3}
	\begin{split}
	\|\mathcal{L}_{\phi}  f \|_{\dot{H}^m}^2 & \lesssim  \rmin^{-2\t/n}  (\|f\|^2_{\dot{H}^{m+\alpha}}+\|\rho\|^2_{\dot{H}^{m+\alpha}}) + p_N(\rmax,\rmin^{-1}, \| \rho\|_{\dot{H}^{m-1+\a}}, \| f\|_{\dot{H}^{2 + \frac{n}{2}}}), \\
	\|\mathcal{L}_{\phi}  f \|_{\dot{H}^m}^2  & \gtrsim   \rmax^{-2\t/n}(\|f\|^2_{\dot{H}^{m+\alpha}}+\|\rho\|^2_{\dot{H}^{m+\alpha}})  - p_N(\rmax,\rmin^{-1}, \| \rho\|_{\dot{H}^{m-1+\a}}, \| f\|_{\dot{H}^{2 + \frac{n}{2}}}).
	\end{split}
	\end{equation}
\end{proposition}

As a consequence of this proposition we obtain control on the key norm $\|\rho\|_{\dH^{m+\a}}$, that will appear in the main estimates on $Y_m$:
\begin{equation}\label{e:rhoY}
\|\rho\|_{\dH^{m+\a}}^2 \lesssim Y_m^N,
\end{equation}
for some large $N \in \N$.  Indeed, setting $f=\rho$ in the above, we find ($N$ may change from line to line)
\begin{multline*}
\|\rho\|_{\dH^{m+\a}}^2 \lesssim \rmax^{2\t/n}	\|\cL_{\phi}  \rho \|_{\dot{H}^m}^2 +   p_N(\rmax,\rmin^{-1}, \| \rho\|_{\dot{H}^{m-1+\a}}) \\
\leq   \rmax^{2\t/n} \| u\|_{\dH^{m+1}}^2 +\rmax^{2\t/n} \|e\|_{\dH^m}^2+ p_N(\rmax,\rmin^{-1}, \| \rho\|_{\dot{H}^{m-1+\a}}) \leq Y_m^4 + p_N(\rmax,\rmin^{-1}, \| \rho\|_{\dot{H}^{m-1+\a}}).
\end{multline*}
Now by the same estimate applied to $ \|\rho\|_{\dot{H}^{m-1+\alpha}}$ we have
\[
 \|\rho\|_{\dot{H}^{m-1+\alpha}}^2 \leq Y_{m-1}^4 + p_N(\rmax,\rmin^{-1}, \| \rho\|_{\dot{H}^{m-2+\a}}).
\]
However, trivially $Y_{m-1} \leq Y_m$ and $ \|\rho\|_{\dot{H}^{m-2+\alpha}} \leq  \|  \rho \|_{\dot{H}^{m}}$ for all $0<\a<2$ with the latter being included into the definition of $Y_m$.  Hence,
\[
\|\rho\|_{\dH^{m+\a}}^2 \lesssim Y_m^4 + p_N(\rmax,\rmin^{-1}, Y_m) \leq Y_m^N,
\]
and \eqref{e:rhoY} follows. 

Conversely, it is clear that $\|\rho\|_{\dH^{m+\a}}$ controls 	$\|\cL_{\phi}  \rho \|_{\dot{H}^m}$ by first in \eqref{e:Lcommmain3}. So, along with $\| u\|_{\dH^{m+1}}^2$ it controls $e$. We obtain
\[
Y_m \sim  \| u\|_{\dH^{m+1}}^2 +  \|\rho\|_{\dH^{m+\a}}^2 + \rmax + \rmin^{-1}.
\]

\begin{remark} Although, as we have just seen, estimate \eqref{e:Lcommmain3} is sufficient to establish control over $\|\rho\|_{\dH^{m+\a}}$, what one can actually prove following our argument below is a somewhat sharper version of \eqref{e:Lcommmain3} where the dependence on the density $\rho$ is of  order below $m+\a$. Namely, for every $\e>0$ there exists a $c_\e >0$ such that
	\begin{equation}\label{e:Lcommmain4}
\begin{split}
\|\mathcal{L}_{\phi}  f \|_{\dot{H}^m}^2 & \lesssim \|f\|^2_{\dot{H}^{m+\alpha}} + \| \rho\|_{\dot{H}^{m-1+\a}}^N  \|f\|_{\dot{H}^{m-1+\alpha}}^2+ \|  \rho \|_{\dot{H}^{m}}^2 \| f\|_{\dot{H}^{2 + \frac{n}{2}}}^2 +c_\e  \|\rho\|_{\dot{H}^{m-1+\a+\e}}^2\|f\|_{\dot{H}^{1+\frac{n}{2}}}^2 , \\
\|\mathcal{L}_{\phi}  f \|_{\dot{H}^m}^2  & \gtrsim  \|f\|^2_{\dot{H}^{m+\alpha}} - \| \rho\|_{\dot{H}^{m-1+\a}}^N \|f\|_{\dot{H}^{m-1+\alpha}}^2 -  \|  \rho \|_{\dot{H}^{m}}^2 \| f\|_{\dot{H}^{2 + \frac{n}{2}}}^2 - c_\e  \|\rho\|_{\dot{H}^{m-1+\a+\e}}^2\|f\|_{\dot{H}^{1+\frac{n}{2}}}^2.
\end{split}
\end{equation}
	Here inequality signs $\lesssim, \gtrsim$ mean up to multiples of $\rmin$ and $\rmax$. 
\end{remark}

As a first step in proving \prop{p:maincoer} we show a basic coercivity estimate. 
\begin{lemma}[Basic coercivity]\label{l:m=0} For any $0<\a<2$  the following bounds hold
\begin{equation}
\begin{split}
\|\mathcal{L}_{\phi}f\|^2_2 & \lesssim \rmin^{-2\t/n} \|f\|^2_{\dot{H}^{\alpha}} +  \rmax^{2\t/n} \rmin^{-2-4\t/n} |\n \rho|^2_\infty  \| f\|^2_{\dH^{\a/2}} \\
\|\mathcal{L}_{\phi}f\|^2_2 & \gtrsim \rmax^{-2\t/n} \|f\|^2_{\dot{H}^{\alpha}} -  \rmax^{2\t/n} \rmin^{-2-4\t/n} |\n \rho|^2_\infty  \| f\|^2_{\dH^{\a/2}}.
\end{split}
\end{equation}
\end{lemma}
\begin{proof}
Let us denote
\[
\fint_{\O(0,z)} \rho(x+\xi) \dxi = \frac{1}{|\O(0,z)|} \int_{\O(0,z)} \rho(x+\xi) \dxi.
\]
Note that $|\O(0,z)| \sim |z|^n$.  In order to remove the $x$-dependence from the kernel we ``freeze"  the coefficient, meaning replace $d$ with the average value and then replace it with $\rho(x)$:
\[
\cL_\phi f(x) = \rho(x)^{-\tau/n} \int_{\mathbb{T}^n} \frac{h(|z|)}{|z|^{n+\alpha}} \d_z f(x)dz+ \int_{\mathbb{T}^n} \frac{h(|z|)}{|z|^{n+\alpha}}\left( \frac{1}{\left[\fint_{\O(0,z)} \rho(x+\xi) \dxi\right]^{\t/n}} - \frac{1}{\rho^{\t/n}(x)} \right) \d_z f(x)dz.
\]
The first integral represents the truncated fractional Laplacian $\L_\a$, and hence is bounded above and below by $\rmin^{-\t/n} \| f\|_{\dH^\a}$ and $\rmax^{-\t/n} \| f\|_{\dH^\a}$, respectively.  In the residual term we estimate 
\[
 \frac{1}{\left[\fint_{\O(0,z)} \rho(x+\xi) \dxi\right]^{\t/n}} - \frac{1}{\rho^{\t/n}(x)} =  \frac{\rho^{\t/n}(x) - \left[\fint_{\O(0,z)} \rho(x+\xi) \dxi\right]^{\t/n}}{\left[\fint_{\O(0,z)} \rho(x+\xi) \dxi\right]^{\t/n} \rho^{\t/n}(x)}
\]
and by Taylor expansion,
\[
\left|\rho^{\t/n}(x) - \left[\fint_{\O(0,z)} \rho(x+\xi) \dxi\right]^{\t/n}\right| \leq \rmax^{\t/n} \rmin^{-1} \left|\rho(x) - \fint_{\O(0,z)} \rho(x+\xi) \dxi \right|  \leq \rmax^{\t/n} \rmin^{-1} |\n \rho|_\infty |z|.
\]
So, the residual term is bounded by
\[
\rmax^{\t/n} \rmin^{-1-2\t/n} |\n \rho|_\infty \int_{\mathbb{T}^n} \frac{h(|z|)}{|z|^{n+\alpha-1}} |\d_z f(x)|dz.
\]
Estimating the $L^2$-norm of the remaining integral for $\a<1$ we get a bound by $\|f\|_2$ by the Minkowskii inequality, and for $\a\geq 1$, 
\[
\left| \int_{\mathbb{T}^n} \frac{h(|z|)}{|z|^{n+\alpha-1}} |\d_z f(x)|dz \right|^2 = \left| \int_{\mathbb{T}^n} \frac{h(|z|)}{|z|^{\frac{n}{2}-\e}}\frac{|\d_z f(x)|}{|z|^{\frac{n}{2} + \a - 1 + \e}} dz \right|^2  \leq C_\e \int_{\mathbb{T}^n} \frac{|\d_z f(x)|^2}{|z|^{n+2(\alpha-1+\e)}} dz.
\]
Integrating in $x$ we obtain $\leq \| f\|_{\dH^{\a-1+\e}}^2$. In either case, we can increase regularity to $ \| f\|_{\dH^{\a/2}}$.

\end{proof}

We now want to lift the base regularity into higher order Sobolev spaces $H^m$. The natural way to obtain such estimates is through a commutator 
\begin{equation}\label{e:Lmdec}
 \partial_i^m \mathcal{L}_{\phi}f =  \mathcal{L}_{\phi} \p_i^m f + [\cL_\phi , \p_i^m] f.
\end{equation}
The commutator can be expanded by the Leibniz rule,
\begin{align*}
[\cL_\phi , \p_i^m] f = \sum_{l=0}^{m-1} {m \choose l}  \mathcal{L}_{\partial_i^{(m-l)}\phi}\partial_i^l f
\end{align*}
The main term in \eqref{e:Lmdec}, upon summation over $i$ enjoys the estimates from \lem{l:m=0}:
\begin{equation}\label{e:Lcommmain}
\begin{split}
\sum_{i=1}^n \|\mathcal{L}_{\phi} \p_i^m f \|_2^2 & \lesssim \rmin^{-2\t/n}  \|f\|^2_{\dot{H}^{m+\alpha}} + \rmax^{2\t/n} \rmin^{-2-4\t/n} |\n \rho|^2_\infty   \|f\|_{\dot{H}^{m+\frac{\a}{2}}}^2, \\
\sum_{i=1}^n \|\mathcal{L}_{\phi} \p_i^m f \|_2^2 & \gtrsim  \rmax^{-2\t/n}  \|f\|^2_{\dot{H}^{m+\alpha}} -  \rmax^{2\t/n} \rmin^{-2-4\t/n} |\n \rho|^2_\infty   \|f\|_{\dot{H}^{m+\frac{\a}{2}}}^2.
\end{split}
\end{equation}
By interpolation and the generalized Young inequality, we further obtain
\begin{equation}\label{e:between}
	\begin{split}
	\rmax^{2\t/n} \rmin^{-2-4\t/n} |\n \rho|^2_\infty   \|f\|_{\dot{H}^{m+\frac{\a}{2}}}^2 &\leq 	\rmax^{2\t/n} \rmin^{-2-4\t/n} |\n \rho|^2_\infty    \|f\|_{\dot{H}^{2+\frac{n}{2}}}^{2\th_{m,n,\a}} \|f\|_{\dot{H}^{m+ \alpha}}^{2-2\th_{m,n,\a}} \\
	& \leq c_\e p_N(\rmax,\rmin^{-1}, \| \rho\|_{\dot{H}^{m-1+\a}}, \| f\|_{\dot{H}^{2 + \frac{n}{2}}}) + \e \rmax^{-2\t/n}  \|f\|_{\dot{H}^{m+ \alpha}}^{2}
	\end{split}
\end{equation}
The highest term $\e \rmax^{-2\t/n}  \|f\|_{\dot{H}^{m+ \alpha}}^{2}$ for small $\e$ can be absorbed into the leading terms in \eqref{e:Lcommmain}. Thus, we obtain  required bounds \eqref{e:Lcommmain3} from the highest term.  The rest follows from the following estimate on the commutator.

\begin{lemma}[Main commutator estimate]\label{l:maincomm}
We have the following inequality
\begin{equation}\label{e:maincomm}
	\| [\cL_\phi , \p_i^m] f \|_2^2 \lesssim  \|  \rho \|_{\dot{H}^{m-1+\a}}^N ( \|  f \|_{\dot{H}^{m - \frac12 +\a}}^2+ \|  f \|_{\dot{H}^{m + \frac{\a}{2}}}^2)+ ( \|  \rho \|_{\dot{H}^{m}}^2+   \|\rho\|_{\dot{H}^{m-\frac12+\a}}^2) \|  f \|_{\dot{H}^{2+ \frac{n}{2}}}^2 .
\end{equation}
for some $N =N (m,n,\a) \in \N$. Here, $\lesssim$ means up to a  factor of $\rmax^a \rmin^{-b}$. 
\end{lemma}
All the terms on the right hand side of \eqref{e:maincomm}  can be treated by interpolation between $H^{m+\a}$ and a lower order metric. A computation similar to \eqref{e:between}, thus, readily implies \eqref{e:Lcommmain3}.

\begin{proof} In the course of this proof all  inequalities are understood up to a factor of $\rmax^a \rmin^{-b}$, where $a,b>0$ may change from line to line.  We omit those factors for the sake of brevity.
	
Let us denote by $R(\rho,f)$ the right hand side of \eqref{e:maincomm}. 
	
We denote for short $\p_i = \p$. To show the commutator is of lower order in $f$ we need obtain bounds on $\| \mathcal{L}_{\p^{m-l} \phi} \p^lf \|_2^2$, for $l \in \{0,...,m-1\}$ but first we expand $\partial^{m-l} \phi$ using Faa di Bruno's Formula.

 Writing $\phi(x,y)$ as $\phi(x,x+z)$, we see that the derivatives fall only on the topological part of the kernel. Thus we have
 \begin{align}
 \partial^{m-l} \phi(x,x+z)&=|z|^{-(n+\alpha-\tau)}h(|z|)\partial^{m-l} d^{-\tau}(x,x+z)\\
 \partial^{m-l} d^{-\tau}(x,x+z)&=\partial^{m-l}\left[ \int_{\Omega(x,x+z)} \rho(\xi)d\xi \right]^{-\tau/n}=\partial^{m-l}[d^n(x,x+z)]^{-\tau/n}
 \end{align}
Denoting $g=d^n$ and $h(g)=g^{-\tau/n}$, then using Faa di Bruno's Formula gives,
 \begin{align}
 \partial^{m-l} d^{-\tau}(x,x+z)& =\sum \frac{(m-l)!}{j_1!1!^{j_1}j_2!2!^{j_2}...j_{m-l}!(m-l)!^{j_{m-l}}}h^{(j_1+...+j_{m-l})}(g)\prod_{k=1}^{m-l} \left(\partial^k g\right)^{j_k}
 \end{align}
where the sum is over all $(m-l)$-tuples of integers $\bj=(j_1,...,j_{m-l})$ satisfying
\begin{align}\label{e:jml}
1j_1+2j_2+...+(m-l)j_{m-l}=m-l
\end{align}
Any term in the commutator takes the form,
\begin{align}
\mathcal{L}_{\p^{m-l}\phi}\p^lf(x) = \int_{\mathbb{T}^n}\frac{h(|z|)}{|z|^{n+\alpha-\tau}}\p^{m-l}[d^{-\tau}(x,x+z)]\delta_z\p^lf(x)dz
\end{align}

Then any term in the derivative will take the form
\begin{align}\label{e:Ij}
I_{\bj}[\p^l f](x):=\int_{\mathbb{T}^n} \frac{h(|z|)}{|z|^{n+\alpha-\tau}}\frac{\prod_{k=1}^{m-l}\left(\int_{\Omega(x,x+z)}\partial^k \rho(\xi)d\xi\right)^{j_k}}{d^{\tau+|\bj| n}(x,x+z)}\delta_z\p^lf(x)\dz
\end{align}
where $|\bj|=\sum_{k=1}^{m-l} j_k$.  

\bigskip
\noindent
{\sc Case $0<\a<1$.}
First, we will look at $\int_{\Omega(x,x+z)}\partial^k \rho(\xi)\dxi$. We estimate it with the use of the  Hardy-Littlewood maximal function:
\[
\left| \int_{\Omega(x,x+z)}\partial^k \rho(\xi)\dxi \right| \leq |z|^n  \frac{1}{|z|^n} \int_{\Omega(x,x+z)} |\partial^k \rho(\xi)|\dxi \leq |z|^n  M[\p^k \rho](x),
\]
where 
\[
M[g](x) = \sup_{r> 0} \frac{1}{r^n} \int_{B_r(x)} |g(\xi)|\dxi .
\]
So,
\[
| I_{\bj}[\p^l f](x) | \leq \prod_{k=1}^{m-l}(M[\p^k \rho](x))^{j_k} \int_{\mathbb{T}^n} h(|z|)|\delta_z\p^lf(x)|\frac{\dz}{|z|^{n+\a}}
\]
To estimate the $L^2$-norm of $I_{\bj}[\p^lf]$ we pick a set of conjugate exponents $p_k, q$ such that
\[
\sum_{k=1}^{m-l} \frac{2j_k}{p_k} + \frac2q = 1
\]
and apply H\"older inequality
\[
\begin{split}
\| I_{\bj}[\p^l f]\|_2^2 & \leq  \prod_{k=1}^{m-l}\| M[\p^k \rho] \|_{{p_k}}^{2j_k} \left(\int_{\T^n}\left( \int_{\mathbb{T}^n} h(|z|)|\delta_z\p^lf(x)|\frac{\dz}{|z|^{n+\a}}  \right)^q \dx \right)^{\frac2q} \\
\intertext{by the classical Hardy-Littlewood inequality,}
& \lesssim \prod_{k=1}^{m-l}\| \p^k \rho \|_{{p_k}}^{2j_k} \left(\int_{\T^n}\left( \int_{\mathbb{T}^n} h(|z|)|\delta_z\p^lf(x)|\frac{\dz}{|z|^{n+\a}}  \right)^q \dx \right)^{\frac2q}\\
& \lesssim \prod_{k=1}^{m-l}\| \p^k \rho \|_{{p_k}}^{2j_k} \| \p^l f\|_{W^{\a+\e,q}}^2\\
\intertext{by the Sobolev embeddings,}
& \leq  \prod_{k=1}^{m-l}\|  \rho \|_{\dot{H}^{k+n( \frac12 - \frac{1}{p_k})}}^{2j_k} \| f\|_{\dot{H}^{l+\a+\e + n(\frac12 - \frac1q)}}^2\\
\intertext{Let us make the following choice of exponents: $p_k = \frac{2m}{k}$, $q = \frac{2m}{l}$. Then }
& \leq   \prod_{k=1}^{m-l}\|  \rho \|_{\dot{H}^{k+\frac{n}{2}( 1 - \frac{k}{m})}}^{2j_k} \| f\|_{\dot{H}^{l+\a+\e + \frac{n}{2}(1- \frac{l}{m})}}^2.
\end{split}
\]
 Examining the regularity of the density norms obtained on the last line, we observe that  for all $k = 1,\ldots,m-1$ we have 
\[
k+\frac{n}{2}( 1 - \frac{k}{m}) \leq m-1 + \a,
\]
provided $m$ is large enough. So, the whole density product becomes bounded by a lower order term for all $l=1,\ldots,m-1$:
\[
\prod_{k=1}^{m-l}\|  \rho \|_{\dot{H}^{k+\frac{n}{2}( 1 - \frac{k}{m})}}^{2j_k} \leq \| \rho\|_{\dot{H}^{m-1+\a}}^N,
\]
for some possibly large $N$ (we take the liberty of changing $N$ from line to line in the sequel). When $l=0$, the product above still satisfies the same estimate for all multi-indeces $\bj$ except one where $k=m$, which can only happen if $\bj = (0,\ldots,0,1)$ due to the restriction given by \eqref{e:jml}.  In this case the density term reaches higher order norm $\|  \rho \|_{\dot{H}^{m}}^2$.

As to the $f$-term, we have  for $l \leq m-2$
\[
  l+\a+\e + \frac{n}{2}(1- \frac{l}{m}) < m-1 + \a,
 \]
which contributes the lower order  term. So, in this case, given the density estimates above, we have
\[
\| I_{\bj}[\p^l f]\|_2^2 \leq   \|  \rho \|_{\dot{H}^{m-1+\a}}^N \|  f \|_{\dot{H}^{m-1+\a}}^2+ \|  \rho \|_{\dot{H}^{m}}^2 \|  f \|_{\dot{H}^{2+ \frac{n}{2}}}^2 \leq R(\rho,f), \quad l=0,\ldots,m-2.
\]

For future reference let us record the estimate for the particular  subcase when $l=0$, $j_m=0$:
\begin{equation}\label{e:Ioj}
\| I_{(j_1,\ldots,j_{m-1},0)}[f]\|_2^2 \leq  \|  \rho \|_{\dot{H}^{m-1+\a}}^N \|  f \|_{\dot{H}^{2+ \frac{n}{2}}}^2.
\end{equation}
For the remaining case of $l=m-1$ we have $k=1$, $j_1 = 1$. So,  as far as regularity of $f$,
 \[
 m-1+\a+\e + \frac{n}{2m} < m+ \a - \frac12,
\]
and hence,
\[
\| I_{(1)}[\p^{m-1} f]\|_2^2 \leq  \|  \rho \|_{\dot{H}^{1+\frac{n}{2}}}^2  \|  f \|_{\dot{H}^{m+\a - \frac12}}^2 \leq R(\rho,f).
\]

The obtained estimates cover all the cases, so in summary we have obtained
\begin{equation}
	\| \mathcal{L}_{\p^{m-l}\phi}\p^lf \|_2^2 \leq R(\rho,f).
\end{equation}
which proves  \eqref{e:maincomm}.

\bigskip
\noindent
{\sc Case $1\leq \a<2$.} This is a more involved case since for the application of the Gagliardo-Sobolevskii norm one has to include the next term in the Taylor finite difference of $f$:  $\delta_z\p^lf(x) - z \cdot \n \p^l f(x)$. We therefore add and subtract that term in the formula for $I_\bj[\p^l f](x)$:
\[
\begin{split}
I_\bj[\p^l f](x)& =\int_{\mathbb{T}^n} \frac{h(|z|)}{|z|^{n+\alpha-\tau}}\frac{\prod_{k=1}^{m-l}\left(\int_{\Omega(x,x+z)}\partial^k \rho(\xi)d\xi\right)^{j_k}}{d^{\tau+|\bj| n}(x,x+z)}[ \delta_z\p^lf(x) - z \cdot \n \p^l f(x) ] \dz \\
&+ \int_{\mathbb{T}^n} \frac{h(|z|)}{|z|^{n+\alpha-\tau}}\frac{\prod_{k=1}^{m-l}\left(\int_{\Omega(x,x+z)}\partial^k \rho(\xi)d\xi\right)^{j_k}}{d^{\tau+|\bj| n}(x,x+z)} z \cdot \n \p^l f(x)  \dz \\
& : = I_{\bj,1}[\p^l f](x)+I_{\bj,2}[\p^l f](x).
\end{split}
\]
The estimate on $I_{\bj,1}[\p^l f]$ goes in exact same way as in the previous case noting that the Gagliardo-Sobolevskii definition applies to smoothness exponents away from the interger values, $2>\a+\e >1$.  In $I_{\bj,2}[\p^l f]$ we symmetrize first
\[
\begin{split}
I_{\bj,2}[\p^l f](x) & =  \n \p^l f(x)  \cdot \int_{\mathbb{T}^n} \frac{h(|z|)}{|z|^{n+\alpha-\tau}}\left[ \frac{\prod_{k=1}^{m-l}\left(\int_{\Omega(x,x+z)}\partial^k \rho(\xi)d\xi\right)^{j_k}}{d^{\tau+|\bj| n}(x,x+z)} -  \frac{\prod_{k=1}^{m-l}\left(\int_{\Omega(x,x-z)}\partial^k \rho(\xi)d\xi\right)^{j_k}}{d^{\tau+|\bj| n}(x,x-z)} \right] z \dz \\
& =   \n \p^l f(x)  \cdot \int_{\mathbb{T}^n} \frac{h(|z|)}{|z|^{n+\alpha-\tau}}\prod_{k=1}^{m-l}\left(\int_{\Omega(x,x+z)}\partial^k \rho(\xi)d\xi\right)^{j_k} \left[ \frac{d^{\tau+|\bj| n}(x,x-z) - d^{\tau+|\bj| n}(x,x+z)}{d^{\tau+|\bj| n}(x,x+z)d^{\tau+|\bj| n}(x,x-z)}\right] z \dz \\
& + \n \p^l f(x)  \cdot \int_{\mathbb{T}^n} \frac{h(|z|)}{|z|^{n+\alpha-\tau}d^{\tau+|\bj| n}(x,x-z)}\left[ \prod_{k=1}^{m-l}\left(\int_{\Omega(x,x+z)}\partial^k \rho(\xi)d\xi\right)^{j_k} - \right.\\
& \ \ \ \ \ \ \ \ \ \ \ \ \ \ \ \ \ \ \ \ \ \ \ \ - \left.  \prod_{k=1}^{m-l}\left(\int_{\Omega(x,x-z)}\partial^k \rho(\xi)d\xi\right)^{j_k} \right] z \dz \\
& = I_{\bj,2,1}[\p^l f](x) + I_{\bj,2,2}[\p^l f](x)
\end{split}
\]
By a straightforward computation,
\[
|d^{\tau+|\bj| n}(x,x-z) - d^{\tau+|\bj| n}(x,x+z)| \leq |\n \rho|_\infty |z|^{\t + |\bj| n +1}.
\]
With this at hand we proceed to estimate $I_{\bj,2,1}[\p^l f](x)$:
\[
\begin{split}
| I_{\bj,2,1}[\p^l f](x) | \leq  | \n \p^l f(x) |  \prod_{k=1}^{m-l}(M[\p^k \rho](x))^{j_k}  \int_{\T^n} h(|z|) \frac{\dz}{|z|^{n+\a - 2}}.
\end{split}
\]
Since $\a<2$, the integral converges. Thus,
\[
\| I_{\bj,2,1}[\p^l f]\|_2^2  \leq \prod_{k=1}^{m-l}\| \p^k \rho \|_{{p_k}}^{2j_k} \| \p^{l+1} f\|_{q}^2 \leq  \prod_{k=1}^{m-l}\|  \rho \|_{\dot{H}^{k+n( \frac12 - \frac{1}{p_k})}}^{2j_k} \| f\|_{\dot{H}^{l+1 + n(\frac12 - \frac1q)}}^2.
\]
Since $l+1 < l+ \a+\e$ by further increasing the smoothness of $f$ the estimate blends with the previous case.   

It remains to estimate $I_{\bj,2,2}[\p^l f](x)$. To do this we must estimate
\begin{align*}
\prod_{k=1}^{m-l}\left(\int_{\Omega(x,x+z)}\partial^k \rho(\xi)\dxi\right)^{j_k}-\prod_{k=1}^{m-l}\left(\int_{\Omega(x,x-z)}\partial^k \rho(\xi)\dxi\right)^{j_k}
\end{align*}
We can rewrite such a difference as
\begin{align*}
\prod_{k=1}^{m-l}a_k^{j_k}-\prod_{k=1}^{m-l}b_k^{j_k}=\sum_{k=1}^{m-l}a_1^{j_1}\cdots a_{k-1}^{j_{k-1}}(a_k^{j_k}-b_k^{j_k})b_{k+1}^{j_{k+1}}\cdots b_{m-l}^{j_{m-l}}
\end{align*}
and furthermore,
\begin{align*}
a_k^{j_k}-b_k^{j_k}=(a_k-b_k)(a_k^{j_k-1}+a_k^{j_k-2}b_k+\cdots +a_kb_k^{j_k-2}+b_k^{j_k-1}).
\end{align*}
We will focus on the main difference $a_k-b_k$, while estimating all other terms with the maximal function like before. We write, letting $s=\a-1+\e<1$,
\begin{align*}
& \int_{\Omega(x,x+z)}\p^k\rho(\xi)\dxi-\int_{\Omega(x,x-z)}\p^k\rho(\xi)d\xi=\int_{\Omega(0,z)} \p^k\rho(x+\xi)-\p^k\rho(x-\xi)\dxi\\
& = \int_{\Omega(0,z)} \frac{\p^k\rho(x+\xi)-\p^k\rho(x-\xi)}{|\xi|^{\frac{n}{p_k}+s}}|\xi|^{\frac{n}{p_k}+s} d\xi \lesssim \left( \int_{\Omega(0,z)} \frac{|\p^k\rho(x+\xi)-\p^k\rho(x - \xi)|^{p_k}}{|\xi|^{n+sp_k}}d\xi\right)^{1/p_k} |z|^{n+s}\\
&:=(D_{s,p_k}\p^k\rho(x))^{1/p_k}|z|^{n+s}
\end{align*}
where $\int D_{s,p}g(x) dx=\|g\|_{W^{s,p}}^p$. Then we can estimate the difference in the products by
\begin{align}
\prod_{k=1}^{m-l}&\left(\int_{\Omega(x,x+z)}\p^k\rho(\xi)\dxi\right)^{j_k}-\prod_{k=1}^{m-l}\left(\int_{\Omega(x,x-z)}\p^k\rho(\xi)\dxi\right)^{j_k}\\
&\lesssim \sum_{k=1}^{m-l} \prod_{\substack{i=1\\ i\neq k}}^{m-l} (M[\p^i\rho](x))^{j_i}(M[\p^k\rho](x))^{j_k-1} (D_{s,p_k} \p^k\rho(x))^{1/p_k}|z|^{|\bj| n+s}\nonumber
\end{align}
Therefore returning to $I_{\bj,2,2}[\p^l f]$, we estimate in $L^2$, using the same Holder conjugates as before,
\begin{align}
&\| I_{\bj,2,2}[\p^l f] \|_2^2 \nonumber\\
& \ \ \ \ \ \ \ \lesssim \int_{\mathbb{T}^n}|\nabla \p^lf(x)|^2\left(\int_{\mathbb{T}^n}\frac{h(|z|)}{|z|^{n+\alpha-1-s}} \dz\right)^2 \times \\
& \ \ \ \ \ \ \ \ \ \ \ \ \times \left(\sum_{k=1}^{m-l} \prod_{\substack{i=1\\ i\neq k}}^{m-l} (M[\p^i\rho](x))^{j_i}(M[\p^k\rho](x))^{j_k-1} (D_{s,p_k} \p^k\rho(x))^{1/p_k}\right)^2 \dx\nonumber\\
& \ \ \ \ \ \ \ \lesssim \|\p^{l+1}f\|^2_q\left(\prod_{\substack{i=1 \\ i \neq k}}^{m-l} \|\p^i \rho\|_{p_i}^{2j_i}\right)\|\p^k \rho\|_{p_k}^{2(j_k-1)}\|\p^k\rho\|_{W^{s,p_k}}^2\nonumber\\
& \ \ \ \ \ \ \ \leq \|f\|_{\dot{H}^{l+1+n(\frac{1}{2}-\frac{1}{q})}}^2\prod_{\substack{i=1 \\ i \neq k}}^{m-l} \|\rho\|_{\dot{H}^{i+n(\frac{1}{2}-\frac{1}{p_k})}}^{2j_i}\|\rho\|_{\dot{H}^{k+n(\frac{1}{2}-\frac{1}{p_k})}}^{2(j_k-1)}\|\rho\|_{\dot{H}^{k+s+n(\frac{1}{2}-\frac{1}{p_k})}}^2\nonumber\\
&\ \ \ \ \ \ \ =\|f\|_{\dot{H}^{l+1+\frac{n}{2}(1-\frac{l}{m})}}^2\prod_{\substack{i=1 \\ i \neq k}}^{m-l} \|\rho\|_{\dot{H}^{i+\frac{n}{2}(1-\frac{i}{m})}}^{2j_i}\|\rho\|_{\dot{H}^{k+\frac{n}{2}(1-\frac{k}{m})}}^{2(j_k-1)}\|\rho\|_{\dot{H}^{k+s+\frac{n}{2}(1-\frac{k}{m})}}^2\nonumber
\end{align}
As before let us examine regularity of the density first. In any case when the top $j$-index vanishes, $j_m=0$, so that $i,k \in \{1,...,m-1\}$ we have 
\begin{align*}
i+\frac{n}{2}(1-\frac{i}{m}) &\leq m-1+\alpha\\
k+s+\frac{n}{2}(1-\frac{k}{m}) &\leq m-1+\a
\end{align*}
if $m$ is large enough.  So, in this case the entire product of densities is controlled by the lower order norm:
\[
\prod_{\substack{i=1 \\ i \neq k}}^{m-l} \|\rho\|_{\dot{H}^{i+\frac{n}{2}(1-\frac{i}{m})}}^{2j_i}\|\rho\|_{\dot{H}^{k+\frac{n}{2}(1-\frac{k}{m})}}^{2(j_k-1)}\|\rho\|_{\dot{H}^{k+s+\frac{n}{2}(1-\frac{k}{m})}}^2 \leq \|\rho\|_{\dot{H}^{m-1 + \a}}^N.
\]
This applies in particular for all $l=1,\ldots,m-1$ and even in the case $l=0$ with $\bj=(j_1,...,j_{m-1},0)$. Note that this also extends \eqref{e:Ioj} to the entire range of $\a$'s, $0<\a<2$.

When $k=m$ which is only attainable at $l=0$, $j_m=1$ case, we are off by $\e$: the product collapses to only one norm $\|\rho\|_{\dot{H}^{m-1+\a+\e}}^2$ while the $f$-term is of low order:
\[
\| I_{\bj,2,2}[f] \|_2^2 \leq  \|\rho\|_{\dot{H}^{m-1+\a+\e}}^2 \|f\|_{\dot{H}^{1+\frac{n}{2}}}^2 \leq  \|\rho\|_{\dot{H}^{m-\frac12+\a}}^2 \|f\|_{\dot{H}^{1+\frac{n}{2}}}^2 \leq R(\rho,f).
\]
Combined with the other $\bj$-indeces, the case $l=0$ altogether gives the estimate above. 

Next, for $l=1,\ldots,m-2$, 
\[
l+1+\frac{n}{2}(1-\frac{l}{m}) \leq m-1+\alpha.
\]
So,
\[
\| I_{\bj,2,2}[f] \|_2^2 \leq \|f\|_{\dot{H}^{m-1+\a}}^2 \|\rho\|_{\dot{H}^{m-1+\a}}^2 \leq R(\rho,f).
\]
For the only remaining case $l=m-1$,  the regularity exponent for $f$ is
\[
 m+\frac{n}{2m} \leq m+\frac{\a}{2},
\]
while the density product is of course of lower than $m-1+\a$ order as elucidated above.  So, we arrive at 
\[
\| I_{(1),2,2}[\p^{m-1} f]\|_2^2 \lesssim \|  \rho \|_{\dot{H}^{m-1+\a}}^N \|  f \|_{\dot{H}^{m+\frac{\a}{2}}}^2  \leq R(\rho,f).
\]
\end{proof}

\section{A priori estimates on the velocity equation}\label{s:u}
The goal of this section is to establish a priori bound
\begin{equation}
\p_t \| u\|_{\dH^{m+1}}^2 \leq C Y_m^N.
\end{equation}
Let us rewrite the velocity equation as
\[
\begin{split}
u_t + u \cdot \n u & = \cC_\phi(u,\rho), \\
\cC_\phi(u,\rho)(x) & = \int_{\T^n} \phi(x,x+z) \d_z u(x)  \rho(x+z) \dz = \cL_\phi(u\rho) - u \cL_\phi \rho.
\end{split}
\]
Let us apply $\p^{m+1}$ and test with $\p^{m+1} u$. We have (dropping integrals signs)
\[
\p_t \| u\|_{\dH^{m+1}}^2  = - \p^{m+1} (u \cdot \n u) \cdot \p^{m+1}  u +  \p^{m+1} \cC_\phi(u,\rho)\cdot \p^{m+1}  u.
\]
The transport term is estimated using the classical commutator estimate
\[
\p^{m+1} (u \cdot \n u) \cdot \p^{m+1}  u = u \cdot \n  (\p^{m+1} u) \cdot \p^{m+1}  u + [\p^{m+1} , u] \n u \cdot \p^{m+1}  u 
\]
Then
\[
u \cdot \n  (\p^{m+1} u) \cdot \p^{m+1}  u = - \frac12 (\n \cdot u) | \p^{m+1}  u|^2 \leq |\n u|_\infty \| u\|_{\dH^{m+1}}^2,
\]
and using  \eqref{e:classcomm} for $f = u$, $g = \n u$, we obtain
\[
|[\p^{m+1} , u] \n u \cdot \p^{m+1}  u  | \leq  |\n u|_\infty \| u\|_{\dH^{m+1}}^2.
\]
Thus,
\[
\p_t \| u\|_{\dH^{m+1}}^2  \leq \| u\|_{\dH^{m+1}}^3 +  \p^{m+1} \cC_\phi(u,\rho)\cdot \p^{m+1}  u.
\]

In the rest of the argument we focus on estimating the commutator term.  So, we expand by the product rule
\begin{equation}
	\p^{m+1} \cC_\phi(u,\rho) = \sum_{k=k_1+k_2 = 0}^{m+1} \frac{(m+1)!}{k_1!k_2!(m+1 - k)!}\cC_{\p^{m+1-k} \phi}(\p^{k_1} u, \p^{k_2} \rho).
\end{equation}
Various term in this expansion will be estimated differently. There is however one end-point term which provides necessary dissipation : 
\begin{equation}\label{e:com+1}
\cC_\phi(\p^{m+1} u,\rho) \cdot \p^{m+1} u \leq  - \frac{\rmin}{|\rho|_\infty^{\t/n}}  \| u\|_{\dH^{m+1+\frac{\a}{2}}}^2.
\end{equation}
Note that this particular term eventually guarantees inclusion of the velocity into class $L^2H^{m+1+\frac{\a}{2}}$.

\bigskip
\noindent
{\sc Case $k=1,\ldots,m$.} 
The bulk of the terms can be estimated simultaneously. Those correspond to the range $k = 1,\ldots, m$. We start by the standard symmetrization:
\[
\begin{split}
\int_{\T^n} \cC_{\p^{m+1-k}  \phi} ( \p^{k_1} u, \p^{k_2} \rho) \cdot \p^{m+1} u  \dx & = \int_{\T^{2n}} \d_z \p^{k_1} u (x) \p^{k_2} \rho (x+z) \p^{m+1} u(x) \p^{m+1-k}  \phi(x,x+z) \dz \dx\\
& = \frac12 \int_{\T^{2n}} \d_z \p^{k_1} u (x) \d_z \p^{k_2} \rho (x) \p^{m+1} u(x) \p^{m+1-k}  \phi(x,x+z) \dz \dx\\
& + \frac12 \int_{\T^{2n}} \d_z \p^{k_1} u (x) \p^{k_2} \rho (x) \d_z \p^{m+1} u(x) \p^{m+1-k}  \phi(x,x+z) \dz \dx\\
& = J_1 + J_2.
\end{split}
\]
In the Faa di Bruno expansion of the kernel $ \p^{m+1-k}  \phi(x,x+z) $ we use obtain a set of terms, again, labeled by $\bj = (j_1,...,j_{m+1-k-l})$ with
	\[
	1j_1+...+(m+1-k-l)j_{m+1-k-l} = m+1-k-l.
	\]
With the use of the Hardy-Littlewood maximal function as before we obtain 
\[
J_1 \leq \sum_{\bj} \int_{\T^{2n}} |\d_z \p^{k_1} u (x) \d_z \p^{k_2} \rho (x) \p^{m+1} u(x)| \prod_{l=1}^{m+1-k}(M[\p^l \rho](x))^{j_l} \frac{\dz}{|z|^{n+\a}} \dx
\]
We pick a set of exponents $q_i = \frac{2(m+1)}{k_i}$, $p_l = \frac{2(m+1)}{l}$:
\[
\frac12+ \frac{1}{q_1} + \frac{1}{q_2} + \sum_{l=1}^{m+1-k} \frac{j_l}{p_l} = 1.
\]
We have
\[
\begin{split}
J_1& \leq \sum_{\bj} \int_{\T^{2n}} \frac{|\d_z \p^{k_1} u (x)|}{|z|^{\frac{n}{q_1}+\frac{\a k}{m} +\e}} \frac{| \d_z \p^{k_2} \rho (x) |}{|z|^{\frac{n}{q_2}+\frac{\a(m-k)}{m}}} \frac{|\p^{m+1} u(x)|}{|z|^{\frac{n}{2}-\e}} \prod_{l=1}^{m+1-k}\frac{(M[\p^l \rho](x))^{j_l}}{|z|^{\frac{n j_l}{p_l}}} \dz \dx\\
& \leq \| u \|_{\dH^{k_1 + \frac{\a k}{m} + \frac{n}{2} \frac{m+1-k_1}{m+1} +\e  }} \| \rho \|_{\dH^{k_2 + \frac{\a (m-k)}{m} + \frac{n}{2} \frac{m+1-k_2}{m+1} }} \| u\|_{ \dH^{m+1} }   \prod_{l=1}^{m+1-k}   \| \rho \|^{j_l}_{\dH^{l + \frac{n}{2} \frac{m+1-l}{m+1} }}
\end{split}
\]
Provided $m$ is large enough and $\e$ is small enough we have
\[
\begin{split}
u&: \quad k_1 + \frac{\a k}{m} + \frac{n}{2} \frac{m+1-k_1}{m+1} +\e  < m+1 + \frac{\a}{2},\\
\rho&: \quad  k_2 + \frac{\a (m-k)}{m} + \frac{n}{2} \frac{m+1-k_2}{m+1} < m+\a\\
\rho&: \quad  l  + \frac{n}{2} \frac{m+1-l}{m+1} < m + \a,
\end{split}
\]
for all $k_1 +k_2 = k$, $l=1,...,m+1-k$, $k=1,...,m$.  Thus,
\[
J_1 \leq Y_m^{N} + \e  \|u \|_{\dH^{m+1+ \frac{\a}{2}}}^2
\]
($N$ will change from line to line). Note that the last term can be hidden into dissipation \eqref{e:com+1}. Moving on to $J_2$,
\[
\begin{split}
J_2& \leq \sum_{\bj} \int_{\T^{2n}} \frac{|\d_z \p^{k_1} u (x)|}{|z|^{\frac{n}{q_1}+\frac{\a}{2} +\e}} \frac{|\p^{k_2} \rho (x) |}{|z|^{\frac{n}{q_2} -\e}} \frac{|\d_z \p^{m+1} u(x)|}{|z|^{\frac{n+\a}{2}}} \prod_{l=1}^{m+1-k}\frac{(M[\p^l \rho](x))^{j_l}}{|z|^{\frac{nj_l}{p_l}}} \dz \dx\\
& \leq \| u \|_{\dH^{k_1 + \frac{\a k}{m} + \frac{n}{2} \frac{m+1-k_1}{m+1} +\e  }} \| \rho \|_{\dH^{k_2 + \frac{\a (m-k)}{m} + \frac{n}{2} \frac{m+1-k_2}{m+1} }} \| u\|_{ \dH^{m+1} }   \prod_{l=1}^{m+1-k}   \| \rho \|^{j_l}_{\dH^{l + \frac{n}{2} \frac{m+1-l}{m+1} }}
\end{split}
\]

 We now examine the remaining end-point cases.

\bigskip
\noindent
{\sc Case $k=0$.} Here we deal with only one term
\[
\cC_{\p^{m+1} \phi}(u,\rho) = \mathcal{L}_{\partial^{(m+1)}\phi}[u\rho]  - u \mathcal{L}_{\partial^{(m+1)}\phi} \rho.
\]
In the Faa di Bruno expansion of the kernel, we single out again the case $\bj = (0,...,0,1)$ from the rest, because in the rest of the cases  $\bj =(j_1,...,j_m,0)$ we do not have to use the commutator structure at all.  Instead we have by \eqref{e:Ioj}, (noting that $m\to m+1$) and the control bound \eqref{e:rhoY},
\[
\begin{split}
\int I_{\bj}[u\rho] \cdot \p^{m+1} u  \dx  &\leq  \|I_{\bj}[u\rho]\|_2^2 \|u \|_{\dH^{m+1}} \leq p_N(\| \rho\|_{\dot{H}^{m+\a}}) \|u\rho \|_{\dot{H}^{2+\frac{n}{2}}}^2 \|u \|_{\dH^{m+1}} \\
&\lesssim 1+ \| \rho\|_{\dot{H}^{m+\a}}^{N} + \|u \|_{\dot{H}^{2+\frac{n}{2}}}^8 + \|u \|_{\dH^{m+1}}^2 \leq Y_m^{N}.
\end{split}
\]
And similarly,
\[
\int u I_{\bj}[\rho] \cdot \p^{m+1} u  \dx  \leq  |u|_\infty \| I_{\bj}[\rho] \|_2^2 \|u \|_{\dH^{m+1}} \leq Y_m^{N}.
\]

Let us consider now the more involved term corresponding to $\bj = (0,...,0,1)$. In this case
\[
\begin{split}
\int ( I_{\bj}[u\rho] - u  I_{\bj}[\rho] )  \cdot \p^{m+1} u  \dx & = \int_{\mathbb{T}^{2n}} \frac{h(|z|) \int_{\Omega(x,x+z)}\partial^{m+1} \rho(\xi)\dxi  }{|z|^{n+\alpha-\tau} d^{\tau+n}(x,x+z) } \d_z u(x) \rho(x+z)  \p^{m+1} u(x)  \dz \dx\\
\intertext{after symmetrization,}
&= \frac12 \int_{\mathbb{T}^{2n}} \frac{h(|z|) \int_{\Omega(x,x+z)}\partial^{m+1} \rho(\xi)\dxi  }{|z|^{n+\alpha-\tau} d^{\tau+n}(x,x+z) } \d_z u(x) \d_z \rho(x)  \p^{m+1} u(x)  \dz \dx  \\
&+  \frac12 \int_{\mathbb{T}^{2n}} \frac{h(|z|) \int_{\Omega(x,x+z)}\partial^{m+1} \rho(\xi)\dxi  }{|z|^{n+\alpha-\tau} d^{\tau+n}(x,x+z) } \d_z u(x)  \rho(x) \d_z \p^{m+1} u(x)  \dz \dx
\end{split}
\]
The highest density term suffers a derivative overload and needs to be reduced:
\[
\begin{split}
\int_{\Omega(x,x+z)}\partial^{m+1} \rho(\xi)\dxi  & = \int_{\p \Omega(x,x+z)}\partial^{m} \rho(\xi ) \nu_\xi \dxi =  \int_{\p \Omega(0,z)}\partial^{m} \rho(x+ \xi ) \nu_\xi \dxi \\
& = |z|^{n-1} \int_{\p \Omega(0,\be_1)}\partial^{m} \rho\left(x+ |z| U_z \th \right) \nu_\th \dth \\
\intertext{where $U_z$ is the orthogonal transformation mapping $\be_1$ to $\hat{z}$,}
& = |z|^{n-1} \int_{\p \Omega(0,\be_1)} \left[\partial^{m} \rho\left(x+ |z| U_z \th   \right)- \p^m \rho (x) \right]  \nu_\th \dth 
\end{split}
\] 
We recover one power of $z$ by $|\d_z u | \leq |z| \| \n u \|_\infty$ and in the first intergal $|\d_z \rho | \leq |z| \| \n \rho \|_\infty$.  Putting together we estimate the integrals by
\[
\begin{split}
& \leq  \| \n u \|_\infty \| \n \rho \|_\infty   \int_{\p \Omega(0,\be_1)} \iint_{\mathbb{T}^{2n}} \frac{h(|z|)  \left|\partial^{m} \rho\left(x+ |z| U_z \th   \right)- \p^m \rho (x) \right|  }{|z|^{\frac{n}{2}+\alpha -\frac12} }   \frac{|\p^{m+1} u(x)|}{|z|^{\frac{n}{2} - \frac12}}  \dz \dx  \dth\\
& +  \| \n u \|_\infty \|\rho\|_\infty  \int_{\p \Omega(0,\be_1)} \iint_{\mathbb{T}^{2n}} \frac{h(|z|)  \left|\partial^{m} \rho\left(x+ |z| U_z \th   \right)- \p^m \rho (x) \right|  }{|z|^{\frac{n}{2}+\frac{\a}{2}} }  \frac{|\d_z \p^{m+1} u(x)|}{|z|^{ \frac{n}{2} + \frac{\a}{2} }}  \dz \dx  \dth\\
& \leq \| \n u \|_\infty \| \n \rho \|_\infty  \|u \|_{\dH^{m+1}} +  \| \n u \|_\infty \|\rho\|_\infty  \|u \|_{\dH^{m+1+ \frac{\a}{2}}} ( D_{2\a - 1} (\p^m \rho ) + D_{\a} (\p^m \rho )),
\end{split}
\]
where 
\[
D_s(g) = \int_{\mathbb{T}^{2n}} \frac{h(z)  \left|g \left(x+ |z| U_z \th   \right)- g (x) \right|^2 }{|z|^{n+s} }  \dz \dx .
\]
By \lem{l:SobU} this expression is bounded by the $H^{\frac{s}{2}}$ norm.  Thus,
\begin{equation}
\int( I_{\bj}[u\rho] - u  I_{\bj}[\rho] )   \cdot \p^{m+1} u  \dx  \leq \e  \|u \|_{\dH^{m+1+ \frac{\a}{2}}}^2 + Y_m^N.
\end{equation}

\bigskip
\noindent
{\sc Case $k=m+1$.}   In this case the kernel gets no derivatives, however, we deal with a total of $m$ terms $\cC_{\phi} ( \p^l u, \p^{m+1-l} \rho)$ for $l=0,\ldots,m$ (note that the case $l=m+1$ yields the dissipative term which has been considered already).  Let us consider first the end-point case of $l=0$.  In this case the density suffers a derivative overload. We apply the following ``easing" technique:
\[
\int_{\T^n} \cC_{\phi} ( u, \p^{m+1} \rho) \cdot \p^{m+1} u  \dx  = \iint_{\T^{2n}} \phi(x,x+z) \d_z u(x) \p^{m+1} \rho(x+z) \p^{m+1} u(x) \dz \dx.
\]
We observe that 
\[
\p^{m+1} \rho(x+z)  = \p_z \p_x^{m} \rho(x+z)  = \p_z ( \p_x^m \rho(x+z)  - \p_x^m \rho(x)) = \p_z \d_z \p^m \rho(x).
\]
Now we integrate by parts in $z$:
\begin{multline*}
\int_{\T^n} \cC_{\phi} ( u, \p^{m+1} \rho) \cdot \p^{m+1} u  \dx	=  \iint_{\T^{2n}} \p_z \phi(x,x+z) \d_z u(x)  \d_z \p^m \rho(x) \p^{m+1} u(x) \dz \dx \ + \\
	+ \iint_{\T^{2n}}  \phi(x,x+z) \p u(x+z)   \d_z \p^m \rho(x) \p^{m+1} u(x) \dz \dx := J_1+ J_2.
\end{multline*}
Let us examine $J_2$ first. By symmetrization,
\[
\begin{split}
J_2 &=  \iint_{\T^{2n}} \d_z \p u(x)   \d_z \p^m \rho(x) \p^{m+1} u(x) \phi \dz \dx -  \iint_{\T^{2n}}  \p u(x)   \d_z \p^m \rho(x) \d_z \p^{m+1} u(x) \phi \dz \dx: = J_{2,1} + J_{2,2} \\
J_{2,1} & \leq \| \n^2 u \|_\infty  \iint_{\T^{2n}} | \d_z \p^m \rho(x)|\frac{\dz}{|z|^{n+\a -1} } | \p^{m+1} u(x) | \dx \leq \| \n^2 u \|_\infty   \|\rho \|_{\dH^{m-1+\a+\e}}  \|u \|_{\dH^{m+1}} \leq Y_m^N,\\
J_{2,2} & \leq  \| \n u \|_\infty  \|\rho \|_{ \dH^{m+\frac{\a}{2}} }  \|u \|_{\dH^{m+1+\frac{\a}{2}}} \leq \e \|u \|_{\dH^{m+1+\frac{\a}{2}}}^2 + Y_m^N.
\end{split}
\]
As to $J_1$, let is first observe that $\p_z \phi(x,x+z)= \psi(x,x+z) $ is antisymmetric, $\psi(x,y) = -\psi(y,x)$. Then, by symmetrization we have
\[
J_1 = \frac12 \iint_{\T^{2n}} \p_z \phi(x,x+z) \d_z u(x)  \d_z \p^m \rho(x) \d_z \p^{m+1} u(x) \dz \dx.
\]
Since  
\[
 \p_z \phi(x,x+z) = -(n+\a-\t) h(z) \frac{z_i}{|z|^{n+\a+2 - \t} d^\t} + h(z) \frac{\p_z \int_{\O(x,x+z)} \rho(\xi) \dxi}{|z|^{n+\a-\t} d^{\t + n}(x,x+z)} +  \frac{\p_z h(z)}{|z|^{n+\a- \t} d^\t}
\]
and noticing that
\[
\left| \p_z \int_{\O(x,x+z)} \rho(\xi) \dxi \right| \leq |\rho|_\infty |z|^{n-1},
\]
we can see that this kernel is of order $|z|^{-n-\a-1}$ up to the usual quantities bounded by $Y_m^N$. The one derivative loss is compensated by $|\d_z u(x)| \leq |z|\| \n u\|_\infty $.  With this at hand we estimate $J_1$:
\[
J_1 \leq Y_m^N \| \n u\|_\infty \|u \|_{\dH^{m+1+\frac{\a}{2}}}\|\rho \|_{ \dH^{m+\frac{\a}{2}} } \leq \e \|u \|_{\dH^{m+1+\frac{\a}{2}}}^2 + Y_m^N.
\]

Let us now examine the rest of the commutators $\cC_{\phi} ( \p^l u, \p^{m+1-l} \rho)$ for $k=1,\ldots, m+1$. After symmetrization we obtain
\begin{multline*}
	\int_{\T^n} \cC_{\phi} ( \p^l u, \p^{m+1-l} \rho) \cdot \p^{m+1} u  \dx = \frac12 \int_{\T^{2n}} \d_z \p^l u(x)   \d_z \p^{m+1-l} \rho(x) \p^{m+1} u(x) \phi \dz \dx+\\
	+ \int_{\T^{2n}} \d_z \p^l u(x)   \p^{m+1-l} \rho(x) \d_z \p^{m+1} u(x) \phi \dz \dx: = J_1 + J_2.
\end{multline*}
For $J_1$ we distribute the singularity of the kernel among the three terms 
\[
J_1 \leq \int_{\T^{2n}} \frac{|\d_z \p^l u(x)|}{|z|^{\frac{n}{p} + \frac{2\a}{q} + \e}}   \frac{|\d_z \p^{m+1-l} \rho(x)| }{|z|^{\frac{n}{q} + \frac{2\a}{p}}} \frac{ |\p^{m+1} u(x)|}{|z|^{\frac{n}{2} - \e}} \dz \dx,
\]
using a H\"older triple
\[
\frac1p + \frac1q + \frac12 = 1.
\]
We have
\[
J_1 \leq \|u \|_{\dW^{l+\e+ \frac{2\a}{q},p }} \|\rho \|_{\dW^{ m+1 - l +\frac{2\a}{p},q }} \|u \|_{\dH^{m+1}} \leq   \|u \|_{\dH^{l+\e+ \frac{2\a}{q} + n( \frac12 - \frac1p)}} \|\rho \|_{\dH^{ m+1 - l + \frac{2\a}{p}+ n( \frac12 - \frac1q) }}\|u \|_{\dH^{m+1}}.
\]
Choosing $p = 2\frac{m+1}{l}$ and $q = 2\frac{m+1}{m+1-l}$ we verify for all $l=1,\ldots,m$
\[
\begin{split}
u&: \quad l+ \e+ \frac{(m+1-l)\a}{m+1} + \frac{n(m+1-l)}{2(m+1)} < m+1+\frac{\a}{2},\\
\rho&: \quad   m+1 - l + \frac{l\a}{m+1}+ \frac{l n}{2(m+1)}  < m+\a,
\end{split}
\]
We conclude as before
 \[
 J_1 \leq  \e \|u \|_{\dH^{m+1+\frac{\a}{2}}}^2 + Y_m^N.
 \]
For $J_2$ the computation is similar:
\begin{multline*}
J_2 \leq \int_{\T^{2n}} \frac{|\d_z \p^l u(x)|}{|z|^{\frac{n}{p} + \e + \frac{\a}{2}}}   \frac{| \p^{m+1-l} \rho(x)| }{|z|^{\frac{n}{q} - \e}} \frac{ |\d_z \p^{m+1} u(x)|}{ |z|^{\frac{n}{2} + \frac{\a}{2} }} \dz \dx \leq  \|u \|_{\dW^{l+\e+ \frac{\a}{2},p }} \|\rho \|_{\dW^{ m+1 - l,q }} \\
\leq  \|u \|_{\dH^{l+\e+ \frac{\a}{2} + n( \frac12 - \frac1p)}} \|\rho \|_{\dH^{ m+1 - l + n( \frac12 - \frac1q) }}  \|u \|_{\dH^{m+1+ \frac{\a}{2}}} \leq \e \|u \|_{\dH^{m+1+\frac{\a}{2}}}^2 + Y_m^N,
\end{multline*}
where the last line follows by the same choice of $p,q$ and noting that
\[
\begin{split}
&u: \quad l+\e + \frac{\a}{2} +  \frac{n(m+1-l)}{2(m+1)} \leq m+1 \\
& \rho : \quad m+1 - l + \frac{l n}{2(m+1)}  \leq m+\a,
\end{split}
\]
for all $l=1,\ldots,m$.

\section{A priori estimates on the $e$-equation}\label{s:e}

Consider the quantity 
\begin{align*}
e=\nabla \cdot u + \mathcal{L}_{\phi}\rho.
\end{align*}
The goal of this section is to show
\begin{align*}
\ddt\|e\|_{\dot{H}^m}^2\leq CY_m^N.
\end{align*}
We have,
\begin{align*}
\rho_t+\nabla \cdot (\rho u)=0
\end{align*}
Due to the topological part of the model, the interaction kernel depends on the density $\rho$. Therefore the operator $\mathcal{L}_{\phi}$ does not commute with derivatives. Taking the divergence of the momentum equation and using the density equation and the $e$-quantity we get the identity
\begin{align*}
e_t+\nabla \cdot (ue)&=(\nabla \cdot u)^2-\mathrm{Tr}(\nabla u)^2+ \partial_t(\mathcal{L}_{\phi}(\rho))+\nabla \cdot \mathcal{L}_{\phi}(\rho u).
\end{align*}

Let us take a closer look the last two terms and work out a more explicit formula. For  the time derivative,
\begin{align*}
\partial_t(\mathcal{L}_{\phi}(\rho))=\mathcal{L}_{\phi}(\rho_t)+\mathcal{L}_{\phi_t}(\rho)
\end{align*}
where,
\begin{align*}
\mathcal{L}_{\phi_t}(\rho)&:=-\frac{\t}{n}\int_{\mathbb{T}^n}\frac{h(|z|)}{|z|^{n+\a-\t}}\frac{\int_{\O(x,x+z)}\rho_t(\xi) \dxi}{d^{\t+n}(x,x+z)}\d_z\rho(x) \dz\\
&=\frac{\t}{n}\int_{\mathbb{T}^n}\frac{h(|z|)}{|z|^{n+\a-\t}}\frac{\int_{\O(x,x+z)}\nabla \cdot (\rho u)(\xi) \dxi}{d^{\t+n}(x,x+z)}\d_z\rho(x) \dz\\
\end{align*}
Then looking at the divergence we have,
\begin{align*}
\nabla \cdot \mathcal{L}_{\phi}(\rho u)=\mathcal{L}_{\phi}(\nabla \cdot (\rho u))+\mathcal{L}_{\nabla \phi \cdot}(\rho u)
\end{align*}
where,
\begin{align*}
\mathcal{L}_{\nabla\phi \cdot}(\rho u)&=\int_{\mathbb{T}^n}\nabla \phi(x,x+z) \cdot \d_z(\rho u)(x) \dz\\
&=-\frac{\tau}{n}\int_{\mathbb{T}^n}\frac{h(|z|)}{|z|^{n+\alpha-\tau}}\frac{\int_{\Omega(x,x+z)} \nabla \rho(\xi) \dxi}{d^{\tau+n}(x,x+z)}\cdot \d_z(\rho u)(x) \dz
\end{align*}
Now using the density equation we see that the first terms in $\partial_t(\mathcal{L}_{\phi}(\rho))$ and $\nabla \cdot \mathcal{L}_{\phi}(\rho u)$ cancel, and becomes,
\begin{equation}\label{e:iden}
e_t+\nabla \cdot (ue)=(\nabla \cdot u)^2-\mathrm{Tr}(\nabla u)^2+\mathcal{L}_{\phi_t}(\rho) +\mathcal{L}_{\nabla \phi \cdot}(\rho u)
\end{equation}

In order to achieve our estimate we apply $\p^m$ to $\eqref{e:iden}$ and test with $\p^m e$.  Estimating the last two terms will be the main technical component of this section.  So, let us make a few quick comments as to the remaining terms.  Dropping integral signs we have for the transport term
\begin{align*}
\p^m(e \n \cdot u) \p^m e + (u \cdot \n \p^m e) \p^m e +[\p^m(u \cdot \n e) - u \cdot \n \p^m e] \p^m e.
\end{align*}
So, it can treated exactly like the similar term in the momentum in the beginning of \sect{s:u}. For $\p^m[(\nabla \cdot u)^2-\mathrm{Tr}(\nabla u)^2] \p^m e$ we have quandratic in $\n u$ expression whose $L^2$-norm breaks into the product estmate of $\|u\|_{H^{m+1} } |\n u|_\infty$. We thus can see that all these terms are bounded by $Y_m^3$. 

We now focus solely on the residual alignment term and start with the "worst" in a sense end point cases.

\bigskip
\noindent
{\sc End-Case 1.}  Here we estimate the worst term when all $m$ derivatives fall on the density to form a derivative of order $m+1$:
\begin{multline*}
I =\int_{\T^n} \left[ \int_{\Omega(0,z)}  \p^m \nabla \rho(x+\xi) \dxi \delta_{z}(\rho u)(x) -  \int_{\Omega(0,z)} \nabla (u \p^m \rho)(x+\xi) \dxi \delta_{z}\rho(x) \right] \times \\
\times  \frac{h(|z|)}{|z|^{n+\alpha-\tau}d^{\tau+n}(x,x+z)} \dz .
\end{multline*}
Integrating by parts inside  the integrals we obtain the expression
\[
 \int_{\p \Omega(0,z)} [ \p^m \rho(x+\xi) \delta_{z}(\rho u)(x) -  (u \p^m \rho)(x+\xi)  \delta_{z}\rho(x) ] \cdot \nu_\xi  \dxi .
 \]
Using that $ \delta_{z}(\rho u)(x) = \d_z \rho(x) u(x) + \rho(x+z) \d_z u(x)$, we write the integrand as
\[
\p^m \rho(x+\xi)\d_z \rho(x) (u(x) - u(x+\xi))  +   \p^m \rho(x+\xi) \d_z \rho(x) \d_z u(x) +  \p^m \rho(x+\xi)  \rho(x) \d_z u(x) .
\]

We focus on the last term which is most difficult. We write 
\[
 \p^m \rho(x+\xi)  \rho(x) \d_z u(x) =  \p^m \rho(x+\xi)  \rho(x) [\d_z u(x) -\n u(x) z] + \p^m \rho(x+\xi)  \rho(x)\n u(x) z.
 \]
We focus on the last term. Let us write the integral to be estimated
\[
J = \int_{\T^n}  \int_{\p \Omega(0,z)} [\p^m \rho(x+\xi) - \p^m\rho(x)]  \rho(x)\n u(x) z \cdot  \nu_\xi  \dxi  \frac{h(|z|)}{|z|^{n+\alpha-\tau}d^{\tau+n}(x,x+z)} \dz 
\]
Changing the variable to $\th \in \p \O(0,\be_1)$ we obtain
\[
J =  \int_{\p \Omega(0,\be_1)} \int_{\T^n}  [\p^m \rho(x+|z|U_z \th) - \p^m\rho(x)]  \rho(x)\n u(x) z \cdot  U_z \nu_\th   \frac{h(|z|)}{|z|^{\alpha-\tau+1}d^{\tau+n}(x,x+z)} \dz  \dth.
\]
Let us freeze the coefficients in the kernel:
\[
J = J_1 + J_2,
\]
where
\[
\begin{split}
J_1 & =  \int_{\p \Omega(0,\be_1)}  \rho^{-\t/n}(x) \int_{\T^n}  [\p^m \rho(x+|z|U_z \th) - \p^m\rho(x)]  \n u(x) z \cdot  U_z \nu_\th   \frac{h(|z|)}{|z|^{n+\alpha+1}}  \dz  \dth\\
J_2 & =  \int_{\p \Omega(0,\be_1)} \int_{\T^n}  [\p^m \rho(x+|z|U_z \th) - \p^m\rho(x)]  \rho(x)\n u(x) z \cdot  U_z \nu_\th   \frac{h(|z|)}{|z|^{n+\alpha+1}} \times \\
&\ \ \ \ \ \ \ \ \ \ \ \ \ \times \left( \frac{1}{\left[\fint_{\O(0,z)} \rho(x+\xi) \dxi\right]^{\t/n+1}} - \frac{1}{\rho^{\t/n+1}(x)} \right)  \dz  \dth
\end{split}
\]

To estimate $J_1$ we further symmetrize in $z$ noting that $U_{-z} = - U_z$, and so the kernel is even:
\begin{multline*}
	J_1 =  \int_{\p \Omega(0,\be_1)}  \rho^{-\t/n}(x) \int_{\T^n}  [\p^m \rho(x+|z|U_z \th) +\p^m \rho(x-|z|U_z \th) - 2 \p^m\rho(x)] \times \\
	\times  \n u(x) z \cdot  U_z \nu_\th   \frac{h(|z|)}{|z|^{n+\alpha+1}}  \dz  \dth,
\end{multline*}
and we estimate
\begin{multline*}
\| J_1\|_2 \leq \rmin^{-\t/n} |\n u|_\infty \sum_{i,j,k} \left\| \int_{\T^n}  [\p^m \rho(\cdot+|z|U_z \th)+\p^m \rho(\cdot-|z|U_z \th) - 2\p^m\rho(\cdot)]    \frac{h(|z|) z_i U_z^{j k }}{|z|^{n+\alpha+1}}  dz  \right\|_2 \\
 \leq \rmin^{-\t/n} |\n u|_\infty \| \rho \|_{\dH^{m+\a}},
\end{multline*}
where the ultimate bound follows from \lem{l:Sob2}.\\

To estimate $J_2$ we note that a similar estimate from before gives
\begin{align*}
\left| \frac{1}{\left[\fint_{\O(0,z)} \rho(x+\xi) \dxi\right]^{\t/n+1}} - \frac{1}{\rho^{\t/n+1}(x)}\right|\leq \rmax^{\frac{\tau}{n}+1}\rmin^{-3-\frac{2\t}{n}}|\nabla \rho|_{\infty}|z|
\end{align*}
Therefore by \lem{l:SobU}
\begin{align*}
|J_2|&\leq \rmax^{\frac{\tau}{n}+2}\rmin^{-3-\frac{2\t}{n}}|\nabla \rho|_{\infty}|\nabla u|_{\infty}\int_{\p\O(0,\be_1)}\int_{\mathbb{T}^n}\frac{|\p^m(x+|z|U_z\theta)-\p^m\rho(x)|}{|z|^{\frac{n}{2}+\a-\frac{1}{2}}}\frac{h(|z|)}{|z|^{\frac{n}{2}-\frac{1}{2}}}\dz \dth\\
\|J_2\|_2&\leq \rmax^{\frac{\tau}{n}+2}\rmin^{-3-\frac{2\t}{n}}|\nabla \rho|_{\infty}|\nabla u|_{\infty}\|\rho\|_{\dot{H}^{m+\a-\frac{1}{2}}}.
\end{align*}

Now to estimate the first term. The integral we need to estimate is
\begin{align*}
I&=\rho(x)\int_{\mathbb{T}^n}\int_{\p\O(0,z)}\p^m\rho(x+\xi)\cdot \nu_{\xi} \dxi \frac{[\delta_zu(x)-\nabla u(x)z]h(|z|)}{|z|^{n+\a-\t}d^{\t+n}(x,x+z)}\dz\\
&=\rho(x)\int_{\mathbb{T}^n}\int_{\p\O(0,\be_1)}|\p^m\rho(x+|z|U_z\th)-\p^m\rho(x)|\cdot U_z\nu_{\th}\dth \frac{[\delta_zu(x)-\nabla u(x)z]h(|z|)}{|z|^{1+\a-\t}d^{\t+n}(x,x+z)}\dz\\
|I|&\leq \rmax |\nabla^2 u|_{\infty}\int_{\mathbb{T}^n}\int_{\p\O(0,\be_1)}\frac{h(|z|)|\p^m\rho(x+|z|U_z\th)-\p^m\rho(x)|}{|z|^{n+\a-1}} \dth\dz
\end{align*}
So estimating in $L^2$ and applying \lem{l:SobU} again, we get,
\begin{align*}
\|I\|_2\leq \rmax|\nabla^2 u|_{\infty}\|\rho\|_{\dot{H}^{m+\a-\frac{1}{2}}}
\end{align*}

Now returning to the first integral in this section, we still need to estimate the first two terms,
\begin{align*}
I_1=\int_{\mathbb{T}^n}\int_{\p\Omega(0,z)}\p^m\rho(x+\xi)\d_z\rho(x)(u(x)-u(x+\xi))\cdot \nu_{\xi}\dxi \frac{h(|z|)}{|z|^{n+\a-\t}d^{\t+n}(x,x+z)}\dz
\end{align*}
To estimate this we add and subtract $\p^m\rho(x)u(x)$ in the integrand to get,
\begin{align*}
I_{11}&=u(x)\int_{\mathbb{T}^n}\int_{\p\Omega(0,z)}(\p^m\rho(x+\xi)-\p^m\rho(x))\cdot \nu_{\xi}\dxi \frac{h(|z|)\d_z\rho(x)}{|z|^{n+\a-\t}d^{\t+n}(x,x+z)}\dz\\
I_{12}&=-\int_{\mathbb{T}^n}\int_{\p\Omega(0,z)}(\p^m\rho(x+\xi)u(x+\xi)-\p^m\rho(x)u(x))\cdot \nu_{\xi}\dxi \frac{h(|z|)\d_z\rho(x)}{|z|^{n+\a-\t}d^{\t+n}(x,x+z)}\dz
\end{align*}
Looking at $I_{11}$ we include the next term in the Taylor finite difference.
\begin{align*}
I_{111}&=u(x)\int_{\mathbb{T}^n}\int_{\p\Omega(0,z)}(\p^m\rho(x+\xi)-\p^m\rho(x))\cdot \nu_{\xi}\dxi \frac{h(|z|)[\d_z\rho(x)-z\nabla \rho(x)]}{|z|^{n+\a-\t}d^{\t+n}(x,x+z)}\dz\\
I_{112}&=u(x)\int_{\mathbb{T}^n}\int_{\p\Omega(0,z)}(\p^m\rho(x+\xi)-\p^m\rho(x))\cdot \nu_{\xi}\dxi \frac{h(|z|)z\nabla\rho(x)}{|z|^{n+\a-\t}d^{\t+n}(x,x+z)}\dz
\end{align*}
Notice that shifting to $\p\O(0,\be_1)$ and symmetrizing makes $I_{111}$ and $I_{112}$ take the same form as $J_1$ above, so \lem{l:Sob2} gives
\begin{align*}
\|I_{11}\|_2\leq \rmin^{\t/n}|\nabla \rho|_{\infty}|u|_{\infty}\|\rho\|_{\dot{H}^{m+\a}}
\end{align*}
Proceeding the same way for $I_{12}$ we get
\begin{align*}
I_{121}&=-\int_{\mathbb{T}^n}\int_{\p\Omega(0,z)}(\p^m\rho(x+\xi)u(x+\xi)-\p^m\rho(x)u(x))\cdot \nu_{\xi}\dxi \frac{h(|z|)[\d_z\rho(x)-z\nabla \rho(x)]}{|z|^{n+\a-\t}d^{\t+n}(x,x+z)}\dz\\
I_{122}&=-\int_{\mathbb{T}^n}\int_{\p\Omega(0,z)}(\p^m\rho(x+\xi)u(x+\xi)-\p^m\rho(x)u(x))\cdot \nu_{\xi}\dxi \frac{h(|z|)z\nabla\rho(x)}{|z|^{n+\a-\t}d^{\t+n}(x,x+z)}\dz
\end{align*}
Shifting to $\p\O(0,\be_1)$, symmetrizing and using \lem{l:Sob2} with $g=u\p^m\rho $ also gives
\begin{align*}
\|I_{12}\|_2\leq \rmin^{\t/n}|\nabla \rho|_{\infty}\|u\p^m\rho \|_{\dot{H}^{\a}}\leq  \rmin^{\t/n}|\nabla \rho|_{\infty}|u|_{\infty}\|\rho\|_{\dot{H}^{m+\a}}.
\end{align*}
The second term in the first integral to estimate is
\begin{align*}
I_2=\int_{\mathbb{T}^n}\int_{\p\Omega(0,z)}\p^m\rho(x+\xi)\cdot \nu_{\xi}\dxi \frac{h(|z|)\d_z\rho(x)\d_zu(x)}{|z|^{n+\a-\t}d^{\t+n}(x,x+z)}\dz\\
\end{align*}
We pick up two powers of $z$ from $\d_z\rho(x)$ and $\d_zu(x)$ to get
\begin{align*}
|I_2|&\leq |\nabla \rho|_{\infty}|\nabla u|_{\infty}\int_{\O(0,e_1)}\int_{\mathbb{T}^n}\frac{h(|z|)|\p^{m}\rho(x+|z|U_z\th)-\p^m\rho(x)|}{|z|^{n+\a-1}}\dz\dth
\end{align*}
Applying Holder's inequality and using \lem{l:SobU} we get
\begin{align*}
\|I_2\|\leq|\nabla\rho|_{\infty}|\nabla u|_{\infty} \|\rho\|_{\dot{H}^{m+\a-\frac{1}{2}}}
\end{align*}

Now let us look at the other endpoint where all $m$ derivatives fall inside the increment $\d_zf$ in the residual terms.

\bigskip
\noindent
{\sc End-Case 2.}  Here we need to combine terms from $\mathcal{L}_{\nabla \phi\cdot}(\rho u)$ and $\mathcal{L}_{\phi_t}(\rho)$ again.

\begin{align}
I=\int_{\mathbb{T}^n}\left[\int_{\Omega(0,z)}\nabla \cdot(\rho u)(x+\xi)-\nabla\rho(x+\xi)\cdot u(x)\dxi\right] \frac{h(|z|)\delta_z \partial^m\rho(x)}{|z|^{n+\alpha-\tau}d^{\tau+n}(x,x+z)} \dz
\end{align}
Expanding $\nabla \cdot (\rho u)=\nabla\rho\cdot u+\rho(\nabla\cdot u)$ we get two terms to be estimated. We focus on the last first.
\begin{align*}
J=\int_{\mathbb{T}^n}\left[\int_{\Omega(0,z)}\rho(x+\xi)(\nabla\cdot u)(x+\xi)\dxi\right] \frac{h(|z|)\delta_z \partial^m\rho(x)}{|z|^{n+\alpha-\tau}d^{\tau+n}(x,x+z)} \dz
\end{align*}
As before we will freeze the coefficients, splitting this into $J=J_1+J_2$ with,
\begin{align*}
J_1&=\frac{\rho(x)(\nabla \cdot u)(x)}{\rho^{\frac{\tau}{n}+1}(x)}\int_{\mathbb{T}^n}\frac{h(|z|)}{|z|^{n+\a}}\d_z\p^m\rho(x)\dz\\
J_2&=\int_{\mathbb{T}^n}\frac{h(|z|)}{|z|^{n+\a}}\left(\frac{\fint_{\O(0,z)}\rho(x+\xi)(\nabla \cdot u)(x+\xi)\dxi}{\left(\fint_{\O(0,z)}\rho(x+\xi)\dxi\right)^{\frac{\tau}{n}+1}}-\frac{\rho(x)(\nabla \cdot u)(x)}{\rho^{\frac{\tau}{n}+1}(x)}\right)\d_z\p^m\rho(x)\dz
\end{align*}
The integral in $J_1$ is the truncated fractional Laplacian, so is bounded by $\rmin^{-\frac{\t}{n}+1}\rmax|\nabla u|_{\infty}\|\rho\|_{\dot{H}^{m+\a}}$. Then for $J_2$ we need to control the difference, by adding and subtracting appropriately.
\begin{align*}
J_{2,1}&=\int_{\mathbb{T}^n}\frac{h(|z|)}{|z|^{n+\a}}\fint_{\O(0,z)}\rho(x+\xi)(\nabla \cdot u)(x+\xi)\dxi\left(\frac{\rho^{\frac{\tau}{n}+1}(x)-\left(\fint_{\O(0,z)}\rho(x+\xi)\dxi\right)^{\frac{\tau}{n}+1}}{\left(\fint_{\O(0,z)}\rho(x+\xi)\dxi\right)^{\frac{\tau}{n}+1}\rho^{\frac{\tau}{n}+1}(x)}\right)\d_z\p^m\rho(x)\dz\\
&\leq \rmax^{\frac{\t}{n}+2}\rmin^{-3-\frac{2\t}{n}}|\nabla u|_{\infty}|\nabla\rho|_{\infty}\int_{\mathbb{T}^n}\frac{h(|z|)}{|z|^{n+\a-1}}|\delta_z\p^m\rho(x)|\dz
\end{align*}
for $\a<1$ estimating in $L^2$ we get a bound by $\|\rho\|_{\dot{H}^m}$ by the Minkowskii inequality, and for $\a\geq 1$ we get a bound by $\|\rho\|_{\dot{H}^{m+\a-1+\e}}$. Then looking at $J_{2,2}$ we get,
\begin{align*}
J_{2,2}&=\int_{\mathbb{T}^n}\frac{h(|z|)}{|z|^{n+\a}}\left(\frac{\fint_{\O(0,z)}\rho(x+\xi)(\nabla \cdot u)(x+\xi)\dxi-\rho(x)(\nabla \cdot u)(x)}{\rho^{\frac{\tau}{n}+1}(x)}\right)\d_z\p^m\rho(x)\dz\\
&\leq \rmin^{-1-\frac{\t}{n}}(|\nabla^2u|_{\infty}\rmax+|\nabla u|_{\infty}|\nabla \rho|_{\infty})\int_{\mathbb{T}^n}\frac{h(|z|)}{|z|^{n+\a-1}}|\delta_z\p^m\rho(x)|\dz
\end{align*}
where we can estimate the integral in the same way as for $J_{2,1}$. Now we still need to estimate the first term from expanding $\nabla \cdot (\rho u)$. The term we need to estimate is
\begin{align}
J&=\int_{\mathbb{T}^n}\left[\int_{\Omega(0,z)}\nabla\rho(x+\xi)\cdot (u(x+\xi)-u(x))\dxi\right] \frac{h(|z|)\delta_z \partial^m\rho(x)}{|z|^{n+\alpha-\tau}d^{\tau+n}(x,x+z)} \dz\\
|J|&\leq |\nabla \rho|_{\infty}|\nabla u|_{\infty}\int_{\mathbb{T}^n}\frac{h(|z|)|\d_z\p^m\rho(x)|}{|z|^{n+\a-1}}\dz\nonumber
\end{align}
which again is bounded by $\|\rho\|_{\dot{H}^m}$ for $\a<1$ and $\|\rho\|_{\dot{H}^{m+\a-1+\e}}$ for $\a\geq 1$.\\

We no longer need to combine terms from the two residual terms so we will now proceed to estimate the remainder of the terms from $\mathcal{L}_{\nabla\phi \cdot}(\rho u)$ and $\mathcal{L}_{\phi_t}(\rho)$ individually. First looking at $\mathcal{L}_{\nabla \phi\cdot}(\rho u)$ we will estimate some of the higher order terms where all $m$ derivatives hit the density, and then combine the rest of the intermediary terms in one estimate.\\

\bigskip
\noindent
{\sc End-Case 3.}  In the previous case  we used $u(x)\delta_z\p^m\rho(x)$ from $\d_z(u\p^m\rho(x))=\d_zu(x)\d_z\p^m\rho(x)+u(x)\d_z\p^m\rho(x)+\p^m\rho(x)\d_zu(x)$, we still need to estimate the other two terms.
\begin{align}
I_1&=\int_{\mathbb{T}^n}\frac{h(|z|)\delta_zu(x)\delta_z\p^m\rho(x)}{|z|^{n+\a-\t}d^{\t+n}(x,x+z)}\int_{\O(x,x+z)}\nabla \rho(x)\dxi\dz\\
|I_1|&\leq |\nabla u|_{\infty}|\nabla \rho|_{\infty}\int_{\mathbb{T}^n}\frac{h(|z|)|\d_z\p^m\rho(x)|}{|z|^{n+\a-1}}\dz\nonumber
\end{align}
Then estimating in $L^2$ the integral is bounded by $\|\rho\|_{\dot{H}^m}$ for $\a<1$ and $\|\rho\|_{\dot{H}^{m+\a-1+\e}}$ for $\a\geq 1$.\\

For the second term we need to look at separately for $\a<1$ and for $\a\geq 1$. First $\a<1$,
\begin{align}
I_2&=\int_{\mathbb{T}^n}\frac{h(|z|)(\delta_zu(x))\p^m\rho(x)}{|z|^{n+\a-\t}d^{\t+n}(x,x+z)}\int_{\O(x,x+z)}\nabla \rho(x)\dxi\dz\\
|I_2|&\leq |\nabla u|_{\infty}|\nabla \rho|_{\infty}|\p^m\rho(x)|\int_{\mathbb{T}^n}\frac{h(|z|)}{|z|^{n+\a-1}}\dz\nonumber
\end{align}
which in $L^2$ is bounded by $\|\rho\|_{\dot{H}^m}$. For $\a\geq 1$ we add and subtract the next Taylor term to get $I_2=I_{21}+I_{22}$
\begin{align*}
I_{21}&=\int_{\mathbb{T}^n}\frac{h(|z|)\p^m\rho(x)}{|z|^{n+\a-\t}d^{\t+n}(x,x+z)}\int_{\O(x,x+z)}\nabla \rho(x)\dxi\left[\delta_zu(x)-z\nabla u(x)\right]\dz\\
I_{22}&=\nabla u(x)\p^m\rho(x)\int_{\mathbb{T}^n}\frac{h(|z|)}{|z|^{n+\a-\t}d^{\t+n}(x,x+z)}\int_{\O(x,x+z)}\nabla \rho(x)\dxi z\dz
\end{align*}
For $I_{21}$ we use $|\d_zu(x)-z\nabla u(x)|\leq |\nabla^2u|_{\infty}|z|^2$ to get
\begin{align*}
|I_{21}|&\leq |\nabla^2 u|_{\infty}|\nabla \rho|_{\infty}|\p^m\rho(x)|\int_{\mathbb{T}^n}\frac{h(|z|)}{|z|^{n+\a-2}}\dz
\end{align*}
which in $L^2$ is bounded by $\|\rho\|_{\dH^m}$ again.
To estimate $I_{22}$ we symmetrize first and split into two parts,
\begin{align*}
I_{22}&=\nabla u(x)\p^m\rho(x)\int_{\mathbb{T}^n}\frac{h(|z|)}{|z|^{n+\a-\t}}\left(\frac{\int_{\O(x,x+z)}\nabla \rho(x)\dxi}{d^{\t+n}(x,x+z)}-\frac{\int_{\O(x,x-z)}\nabla \rho(x)\dxi}{d^{\t+n}(x,x-z)}\right) z\dz\\
&=\nabla u(x)\p^m\rho(x)\int_{\mathbb{T}^n}\frac{h(|z|)}{|z|^{n+\a-\t}}d^{-\t-n}(x,x+z)\left(\int_{\O(x,x+z)}\nabla \rho(x)\dxi-\int_{\O(x,x-z)}\nabla \rho(x)\dxi\right) z\dz\\
& \ \ \ \ \ \ +\nabla u(x)\p^m\rho(x)\int_{\mathbb{T}^n}\frac{h(|z|)}{|z|^{n+\a-\t}}\int_{\O(0,z)}\nabla\rho(\xi)\dxi\left(\frac{d^{\t+n}(x,x+z)-d^{\t+n}(x,x-z)}{d^{\t+n}(x,x+z)d^{\t+n}(x,x-z)}\right) z\dz\\
&=I_{221}+I_{222}
\end{align*}

Now for $I_{221}$ we notice that a similar computation as before gives,
\begin{align*}
\int_{\Omega(x,x+z)} \nabla \rho(\xi) \dxi-\int_{\Omega(x,x-z)} \nabla \rho(\xi) \dxi\lesssim (D_{s,p}\p\rho(x))^{1/p}|z|^{n+s}
\end{align*}
where $s=\a-1+\e<1$, and so $n+\a-1-s<n$,
\begin{align*}
|I_{221}|\leq |\nabla u|_{\infty}|\p^m\rho(x)|(D_{s,p}\p\rho(x))^{1/p}\int_{\mathbb{T}^n}\frac{h(|z|)}{|z|^{n+\alpha-1-s}}\dz
\end{align*}
Then using Holder's inequality in $L^2$ with $\frac{2}{p}+\frac{2}{q}=1$ we get,
\begin{align*}
\|I_{221}\|_2^2\leq |\nabla u|_{\infty}^2 \|\rho\|_{\dot{H}^{m+n(\frac{1}{2}-\frac{1}{q})}}^2\|\rho\|_{\dot{H}^{1+s+n(\frac{1}{2}-\frac{1}{p})}}^2
\end{align*}
Then choosing $q=\frac{2m}{m-1}$ and $p=2m$ gives 
\begin{align*}
&\rho : \ \ m+n(\frac{1}{2}-\frac{1}{q})=m+\frac{n}{2m}<m+\a\\
&\rho : \ \ 1+s+n(\frac{1}{2}-\frac{1}{p})=1+s+\frac{n}{2}\frac{m-1}{m}<m+\a
\end{align*}
For $I_{222}$ we have already shown how to estimate the difference $d^{\t+n}(x,x+z)-d^{\t+n}(x,x-z)$ so we get,
\begin{align*}
|I_{222}|&\leq |\nabla u|_{\infty}|\p^m \rho(x)||\nabla \rho|_{\infty}^2 \int_{\mathbb{T}^n}\frac{h(|z|)}{|z|^{n+\a-2}}\dz\\
\|I_{222}\|_2^2&\leq |\nabla u|^2_{\infty}|\nabla \rho|_{\infty}^4\|\rho\|^2_{\dot{H}^m}
\end{align*}

\bigskip
\noindent
{\sc End-Case 4.}  Since $\p^m(\rho u)=\p^{m-1}(\rho \p u)+u\p^m\rho$, we still need to estimate the term
\begin{align}
I_{0,0}[\p^{m-1}(\rho \p u)](x)=\int_{\mathbb{T}^n}\frac{h(|z|)}{|z|^{n+\a-\t}}\frac{\int_{\O(x,x+z)}\nabla \rho(\xi) \dxi}{d^{\t+n}(x,x+z)}\d_z(\p^{m-1}(\rho \p u))(x)\dz
\end{align}
For $\a<1$ and $\e$ so that $\a+\e<1$, we get
\begin{align*}
|I_{0,0}[\p^{m-1}(\rho \p u)](x)|&\leq |\nabla \rho|_{\infty}\int_{\mathbb{T}^n}\frac{h(|z|)}{|z|^{n+\a}}|\d_z(\p^{m-1}(\rho \p u))(x)| \dz\\
\|I_{0,0}[\p^{m-1}(\rho \p u)]\|_2^2&\leq |\nabla \rho|_{\infty}^2\|\rho \p u\|^2_{\dot{H}^{m-1+\a+\e}}\\
&\leq |\nabla \rho|_{\infty}^2\|\rho \|^2_{\dot{H}^{m-1+\a+\e}}\|u\|^2_{\dot{H}^{m+1}}
\end{align*}
For $\a\geq 1$, we again add and subtract the next Taylor term, focus on the second one, and symmetrize
\begin{align*}
I_{0,0,2}[\p^{m-1}(\rho \p u)](x)&=\nabla(\p^{m-1}(\rho \p u))(x)\int_{\mathbb{T}^n} \frac{h(|z|)}{|z|^{n+\a-\t}}\left(\frac{\int_{\O(x,x+z)}\nabla \rho(\xi)\dxi}{d^{\tau+n}(x,x+z)}-\frac{\int_{\O(x,x-z)}\nabla \rho(\xi)\dxi}{d^{\tau+n}(x,x-z)}\right)z\dz
\end{align*}
splitting this into two parts
\begin{align*}
I_{0,0,2,1}[\p^{m-1}(\rho \p u)](x)&=\nabla(\p^{m-1}(\rho \p u))(x)\times\\
&\times\int_{\mathbb{T}^n} \frac{h(|z|)}{|z|^{n+\a-\t}d^{\tau+n}(x,x+z)}\left(\int_{\O(x,x+z)}\nabla \rho(\xi)\dxi-\int_{\O(x,x-z)}\nabla \rho(\xi)\dxi\right)z\dz\\
I_{0,0,2,2}[\p^{m-1}(\rho \p u)](x)&=\nabla(\p^{m-1}(\rho \p u))(x)\times\\
&\times\int_{\mathbb{T}^n} \frac{h(|z|)\int_{\O(x,x-z)}\nabla \rho(\xi)\dxi}{|z|^{n+\a-\t}}\left(\frac{1}{d^{\tau+n}(x,x+z)}-\frac{1}{d^{\tau+n}(x,x-z)}\right)z\dz
\end{align*} 
and estimating these as before we get,
\begin{align*}
|I_{0,0,2,1}[\p^{m-1}(\rho \p u)](x)|&\leq |\nabla \p^{m-1}(\rho \p u)(x)|(D_{s,p}\p \rho(x))^{1/p}\int_{\mathbb{T}^n}\frac{h(|z|)}{|z|^{n+\a-1-s}}\dz\\
\|I_{0,0,2,1}[\p^{m-1}(\rho \p u)]\|_2^2&\leq \|\nabla \p^{m-1}(\rho \p u)\|_q^2 \|\p \rho\|_p^2\\
&\leq \|\rho\|^2_{\dot{H}^{m+n(\frac{1}{2}-\frac{1}{p})}}\|u\|^2_{\dot{H}^{m+1+n(\frac{1}{2}-\frac{1}{p})}}\|\rho\|^2_{\dot{H}^{1+s+n(\frac{1}{2}-\frac{1}{q})}}
\end{align*}
Choosing $q=\frac{2m}{m-1}$ and $p=2m$ we get
\begin{align*}
\|J_{2,1}\|_2^2&\leq\|\rho\|^4_{\dot{H}^{m+\a}}\|u\|_{\dot{H}^{m+1+\frac{n}{2m}}}\\
&\leq Y_m^N\|u\|_{\dot{H}^{m+1}}+\e\|u\|^2_{\dot{H}^{m+1+\frac{\a}{2}}}
\end{align*}
and for $I_{0,0,2,2}[\p^{m-1}(\rho \p u)](x)$ we get,
\begin{align*}
|I_{0,0,2,2}[\p^{m-1}(\rho \p u)](x)|&\leq |\nabla\p^{m-1}(\rho \p u)(x)||\nabla \rho|_{\infty}\int_{\mathbb{T}^n}\frac{h(|z|)}{|z|^{n+\a-2}} \dz\\
\|I_{0,0,2,2}[\p^{m-1}(\rho \p u)]\|_2^2&\leq |\nabla \rho|_{\infty}^2\|\rho\|_{\dot{H}^m}^2\|u\|_{\dot{H}^{m+1}}^2.
\end{align*}

\bigskip
\noindent
{\sc End-Case 5.}  All $m$ derivatives on $\nabla \phi$, and $l =0,...,m-1$. Have to estimate
\begin{align*}
I_{\bj,l}[\rho u](x)=\int_{\mathbb{T}^n}\frac{h(|z|)}{|z|^{n+\alpha-\tau}} \int_{\Omega(x,x+z)}\partial^l \nabla \rho(\xi) \dxi \frac{\prod_{k=1}^{m-l} \left( \int_{\Omega(x,x+z)} \partial^k \rho(\xi) d\xi\right)^{j_k}}{d^{\tau+(|\bj|+1)n}} \cdot \delta_z(\rho u)(x) \dz
\end{align*}
Using the maximal functions we get,
\begin{align*}
|I_{\bj,l}[\rho u](x)|\leq \prod_{k=1}^{m-l} (M[\partial^k \rho](x))^{j_k} M[\partial^{l+1}\rho](x)\int_{\mathbb{T}^n} \frac{h(|z|)}{|z|^{n+\alpha}}|\delta_z\rho u(x)|\dz
\end{align*}
Then using Holder's inequality with
\begin{align*}
\sum_{k=1}^{m-l} \frac{2j_k}{p_k}+\frac{2}{q_1}+\frac{2}{q_2}=1,
\end{align*}
and the Hardy-Littlewood inequality, we get for $0<\a<1,$
\begin{align*}
\|I_{\bj,l}[\rho u]\|_2^2 &\lesssim \prod_{k=1}^{m-l}\|\partial^k \rho\|_{p_k}^{2j_k}\|\partial^{l+1}\rho\|_{q_1}^2\|\rho u\|_{W^{\alpha+\e,q_2}}^2\\
&\leq \prod_{k=1}^{m-l} \|\rho\|_{\dot{H}^{k+n(\frac{1}{2}-\frac{1}{p_k})}}^{2j_k}\|\rho\|_{\dot{H}^{l+1+n(\frac{1}{2}-\frac{1}{q_1})}}^2\|\rho u\|^2_{\dot{H}^{\alpha+\varepsilon+n(\frac{1}{2}-\frac{1}{q_2})}}
\end{align*}
Now we choose, for $l\neq 0$, $p_k=\frac{2m}{k}$, $q_1=\frac{2m-1}{l}$, and $q_2=\frac{2m(2m-1)}{l}$. Then we get
\begin{align*}
\rho: \ \ \ \ k+n\left(\frac{1}{2}-\frac{1}{p_k}\right)=k+\frac{n}{2}(\frac{m-k}{m})\leq m+\alpha
\end{align*}
for $m$ large enough, for all $k=1,...,m$. Then for $l=1,...,m-1$,
\begin{align*}
&\rho: \ \ \ \ l+1+n\left(\frac{1}{2}-\frac{1}{q_1}\right)=l+1+n\left(\frac{2m-1-2l}{2(2m-1)}\right)\leq m+\alpha\\
&\rho u: \ \ \ \ \alpha+\varepsilon+n\left(\frac{1}{2}-\frac{1}{q_2}\right)=\alpha+\varepsilon+\frac{n}{2}\left(\frac{2m^2-m-l}{2m^2-m}\right)<2+\frac{n}{2}
\end{align*}
and for $l=0$, instead of using the maximal function on $\nabla \rho$ we simply estimate with $|\nabla \rho|_{\infty}$ and use $|\delta_z(\rho u)|\leq |z|(|\nabla \rho|_{\infty}|u|_{\infty}+\rmax|u|_{\infty})$ to get
\begin{align*}
|I_{\bj,0}[\rho u](x)|&\leq \prod_{k=1}^m(M[\p^k\rho](x))^{j_k}|\nabla \rho|_{\infty}(|\nabla \rho|_{\infty}|u|_{\infty}+\rmax|\nabla u|_{\infty})\\
\|I_{\bj,0}[\rho u]\|_2^2&\leq \prod_{k=1}^m\|\rho\|_{\dot{H}^k}^{2j_k}|\nabla \rho|^2_{\infty}(|\nabla \rho|_{\infty}|u|_{\infty}+\rmax|\nabla u|_{\infty})^2
\end{align*}
For $\a\geq 1$ we add and subtract the next Taylor term to get $I_{\bj,l}[\rho u]=I_{\bj,l,1}[\rho u]+I_{\bj,l,2}[\rho u]$
\begin{align*}
I_{\bj,l,1}[\rho u](x)&=\int_{\mathbb{T}^n}\frac{h(|z|)}{|z|^{n+\alpha-\tau}}  \frac{\prod_{k=1}^{m-l} \left( \int_{\Omega(x,x+z)} \partial^k \rho(\xi) \dxi\right)^{j_k}}{d^{\tau+(|\bj|+1)n}(x,x+z)}\times\\
& \ \ \ \ \ \ \ \times \left(\int_{\Omega(x,x+z)}\partial^l \nabla \rho(\xi) \dxi\right) \left[\delta_z(\rho u)(x)-z\nabla (\rho u)(x)\right] \dz\\
I_{\bj,l,2}[\rho u](x)&=\nabla (\rho u)(x)\int_{\mathbb{T}^n}\frac{h(|z|)}{|z|^{n+\alpha-\tau}} \int_{\Omega(x,x+z)}\partial^l \nabla \rho(\xi) \dxi \frac{\prod_{k=1}^{m-l} \left( \int_{\Omega(x,x+z)} \partial^k \rho(\xi) \dxi\right)^{j_k}}{d^{\tau+(|\bj|+1)n}(x,x+z)} z \dz
\end{align*}
The argument for $I_{\bj,l,1}[\rho u]$ goes just as above, noting again that the Gagliardo-Sobolevskii definition applies to smoothness exponents away from the integer values, $2>\a+\e>1$. Looking at $I_{\bj,l,2}[\rho u]$ we symmetrize and split further into three parts getting,
\begin{align*}
I_{\bj,l,2}[\rho u](x)&=\nabla (\rho u)(x)\int_{\mathbb{T}^n}\frac{h(|z|)}{|z|^{n+\alpha-\tau}}\Big[ \int_{\Omega(x,x+z)}\partial^l \nabla \rho(\xi) \dxi \frac{\prod_{k=1}^{m-l} \left( \int_{\Omega(x,x+z)} \partial^k \rho(\xi) \dxi\right)^{j_k}}{d^{\tau+(|\bj|+1)n}(x,x+z)}\\
& \ \ \ \ \ \ \ -\int_{\Omega(x,x-z)}\partial^l \nabla \rho(\xi) \dxi \frac{\prod_{k=1}^{m-l} \left( \int_{\Omega(x,x-z)} \partial^k \rho(\xi) \dxi\right)^{j_k}}{d^{\tau+(|\bj|+1)n}(x,x-z)}\Big] z \dz\\
&=\nabla (\rho u)(x)\int_{\mathbb{T}^n}\frac{h(|z|)}{|z|^{n+\alpha-\tau}} \int_{\Omega(x,x+z)}\partial^l \nabla \rho(\xi) \dxi \prod_{k=1}^{m-l} \left( \int_{\Omega(x,x+z)} \partial^k \rho(\xi) \dxi\right)^{j_k}\times\\
& \ \ \ \ \ \ \ \times \left(d^{-\tau-(|\bj|+1)n}(x,x+z)-d^{-\tau-(|\bj|+1)n}(x,x-z)\right)z\dz\\
&+\nabla (\rho u)(x)\int_{\mathbb{T}^n}\frac{h(|z|)}{|z|^{n+\alpha-\tau}} \int_{\Omega(x,x+z)}\partial^l \nabla \rho(\xi) \dxi d^{-\tau-(|\bj|+1)n}(x,x-z) \times\\
& \ \ \ \ \ \ \ \ \times\left(\prod_{k=1}^{m-l} \left( \int_{\Omega(x,x+z)} \partial^k \rho(\xi) \dxi\right)^{j_k}-\prod_{k=1}^{m-l} \left( \int_{\Omega(x,x-z)} \partial^k \rho(\xi) \dxi\right)^{j_k}\right) z\dz\\
&+\nabla (\rho u)(x)\int_{\mathbb{T}^n}\frac{h(|z|)}{|z|^{n+\alpha-\tau}} \prod_{k=1}^{m-l} \left( \int_{\Omega(x,x-z)} \partial^k \rho(\xi) \dxi\right)^{j_k} d^{-\tau-(|\bj|+1)n}(x,x-z) \times\\
& \ \ \ \ \ \ \ \ \times\left(\int_{\Omega(x,x+z)}\partial^l \nabla \rho(\xi) \dxi-\int_{\Omega(x,x-z)}\partial^l \nabla \rho(\xi) \dxi \right) z\dz\\
&=I_{\bj,l,2,1}[\rho u](x)+I_{\bj,l,2,2}[\rho u](x)+I_{\bj,l,2,3}[\rho u](x)
\end{align*}
For $I_{\bj,l,2,1}[\rho u](x)$ and $I_{\bj,l,2,2}[\rho u](x)$ we make the same estimates as before, using
\begin{align*}
|d^{\tau+(|\bj|+1)n}(x,x+z)-d^{\tau+(|\bj|+1)n}(x,x-z)|\leq |\nabla \rho|_{\infty}|z|^{\t+(\bj+1)n+1}
\end{align*}
and also applying the Maximal function to $\int_{\O(x,x+z)}\nabla \p^l\rho(\xi)\dxi \leq |z|^nM[\p^{l+1}(\rho)](x)$ to get,
\begin{align*}
|I_{\bj,l,2,1}[\rho u](x)|&\leq |\nabla (\rho u)|_{\infty}|\nabla \rho|_{\infty} M[\p^{l+1}\rho](x)\prod_{k=1}^{m-l}(M[\p^k\rho](x))^{j_k}\times\\
& \ \ \ \ \ \ \ \times\left(\int_{\mathbb{T}^n}\frac{h(|z|)}{|z|^{n+\a-2}}\dz\right)\\
\|I_{\bj,l,2,1}[\rho u]\|_2^2&\leq |\nabla (\rho u)|^2_{\infty}|\nabla \rho|^2_{\infty}\|\p^{l+1}\rho\|_q^2\prod_{k=1}^{m-l}\|\p^k\rho\|_{p_k}^{2j_k}\\
&\leq |\nabla (\rho u)|^2_{\infty}|\nabla \rho|^2_{\infty}\|\rho\|_{\dot{H}^{l+1+{n}(\frac{1}{2}-\frac{1}{q})}}^2\prod_{k=1}^{m-l}\|\rho\|_{\dot{H}^{k+n(\frac{1}{2}-\frac{1}{p_k})}}^{2j_k}
\end{align*}
where we used Holder's inequality with
\begin{align*}
\sum_{k=1}^{m-l}\frac{2j_k}{p_k}+\frac{2}{q}=1
\end{align*}
Picking $q=\frac{2m}{l}$ and $p_k=\frac{2m}{k}$ gives
\begin{align*}
&\rho : \ \ \ l+1+\frac{n(m-l)}{2m}\leq m+\a\\
&\rho : \ \ \ k+\frac{n(m-k)}{2m}\leq m+\a
\end{align*}
Then for $I_{\bj,l,2,2}[\rho u](x)$ we get
\begin{align*}
|I_{\bj,l,2,2}[\rho u](x)|&\leq |\nabla (\rho u)|_{\infty}M[\p^{l+1}\rho](x)\sum_{k=1}^{m-l}\prod_{\substack{i=1 \\ i\neq k}}^{m-l} (M[\p^i\rho](x))^{j_i}M([\p^k \rho](x))^{j_k-1}(D_{s,p_k}\p^k\rho(x))^{1/p_k}\times\\
& \ \ \ \ \ \ \ \ \ \ \ \ \ \ \times\left(\int_{\mathbb{T}^n}\frac{h(|z|)}{|z|^{n+\a-1-s}}\dz\right)\\
\|I_{\bj,l,2,2}[\rho u]\|_2^2&\leq |\nabla (\rho u)|_{\infty}^2\|\p^{l+1}\rho\|_{q}^2\prod_{\substack{i=1 \\ i \neq k}}^{m-l}\|\p^i\rho\|_{p_i}^{2j_i}\|\p^k\rho\|_{p_k}^{2(j_k-1)}\|\p^k\rho\|_{W^{s,p_k}}^2\\
&\leq  |\nabla (\rho u)|_{\infty}^2\|\p^{l+1}\rho\|_{\dot{H}^{l+1+n(\frac{1}{2}-\frac{1}{q})}}^2\prod_{\substack{i=1 \\ i \neq k}}^{m-l}\|\rho\|_{\dot{H}^{i+n(\frac{1}{2}-\frac{1}{p_i})}}^{2j_i}\|\rho\|_{\dot{H}^{k+n(\frac{1}{2}-\frac{1}{p_k})}}^{2(j_k-1)}\|\rho\|_{\dot{H}^{k+s+n(\frac{1}{2}-\frac{1}{p_k})}}^2
\end{align*}
picking the same Holder conjugates gives
\begin{align*}
&\rho : \ \ \ l+1+\frac{n(m-l)}{2m}\leq m+\a\\
&\rho : \ \ \ k+\frac{n(m-k)}{2m}\leq m+\a\\
&\rho : \ \ \ k+s+\frac{n(m-k)}{2m}\leq m+\a
\end{align*}
To estimate $I_{\bj,l,2,3}[\rho u](x)$ we note that a similar computation as before gives
\begin{align*}
\int_{\Omega(x,x+z)}\partial^l \nabla \rho(\xi) \dxi-\int_{\Omega(x,x-z)}\partial^l \nabla \rho(\xi) \dxi\leq |z|^{n+s}(D_{s,q}\p^{l+1}\rho(x))^{1/q}
\end{align*}
therefore, again using the Maximal function we get,
\begin{align*}
|I_{\bj,l,2,3}[\rho u](x)|&\leq |\nabla (\rho u)|_{\infty}(D_{s,q}\p^{l+1}\rho(x))^{1/q}\prod_{k=1}^{m-l}(M[\p^k\rho](x))^{j_k}\int_{\mathbb{T}^n}\frac{h(|z|)}{|z|^{n+\a-1-s}}\dz\\
\|I_{\bj,l,2,3}[\rho u]\|_2^2&\leq |\nabla (\rho u)|_{\infty}^2\|\rho\|^2_{\dot{H}^{l+1+s+n(\frac{1}{2}-\frac{1}{q})}}\prod_{k=1}^{m-l}\|\rho\|^{2j_k}_{\dot{H}^{k+n(\frac{1}{2}-\frac{1}{p_k})}}
\end{align*}
choosing the same Holder conjugates gives,
\begin{align*}
&\rho : \ \ \ l+1+s+\frac{n(m-l)}{2m}\leq m+\a\\
&\rho : \ \ \ k+\frac{n(m-k)}{2m}\leq m+\a
\end{align*}

\bigskip
\noindent
{\sc Intermediary  Cases.}  For all $l=1,...,m-1, i=0,...,m-l,$ and $k=1,...,m-l-i$, we have to estimate

\begin{align*}
I_{\bj,i}[\p^l(\rho u)](x)=\int_{\mathbb{T}^n} \frac{h(|z|)}{|z|^{n+\alpha-\tau}}\int_{\Omega(x,x+z)}\partial^i\nabla \rho(\xi)\dxi \frac{\prod_{k=1}^{m-l-i}\left( \int_{\Omega(x,x+z)}\partial^k \rho(\xi)\dxi\right)^{j_k}}{d^{\tau+(|j|+1)n}(x,x+z)}\cdot \delta_z\partial^l(\rho u)(x)\dz
\end{align*}
First, for $0<\a<1$, we employ the Maximal functions again to get,
\begin{align*}
|I_{\bj,i}[\p^l(\rho u)](x)|\lesssim \prod_{k=1}^{m-l-i}\left(M[\partial^k\rho](x)\right)^{j_k} M[\partial^{i+1}\rho](x)\int_{\mathbb{T}^n}\frac{h(|z|)}{|z|^{n+\alpha}}|\delta_z\partial^l (\rho u)(x)|\dz
\end{align*}
Then estimating in $L^2$-norm, applying Holder's inequality with
\begin{align*}
\sum_{k=1}^{m-l-i}\frac{2j_k}{p_k}+\frac{2}{q_1}+\frac{2}{q_2}=1,
\end{align*}
and using the Hardy-Littlewood inequality, we get
\begin{align*}
\|I_{\bj,i}[\p^l(\rho u)]\|_2^2\lesssim \prod_{k=1}^{m-l-i}\|\rho\|_{\dot{H}^{k+n(\frac{1}{2}-\frac{1}{p_k})}}^{2j_k}\|\rho\|^2_{\dot{H}^{i+1+n(\frac{1}{2}-\frac{1}{q_1})}}\|\rho u\|_{\dot{H}^{l+\alpha+\varepsilon+n(\frac{1}{2}-\frac{1}{q_2})}}^2
\end{align*}
Now we choose $p_k=\frac{2m}{k}, q_1=\frac{2m}{i}$, and $q_2=\frac{2m}{l}$. Provided $m$ is large enough and $\varepsilon$ is small enough,

\begin{align*}
\rho&: \ \ \ k+\frac{n}{2}\left(\frac{m-k}{m}\right)\leq m-1+\alpha\\
\rho&: \ \ \ i+1+\frac{n}{2}\left(\frac{m-i}{m}\right)\leq m+\alpha\\
\rho u&: \ \ \ l+\alpha+\varepsilon+\frac{n}{2}\left(\frac{m-l}{m}\right)\leq m+\alpha
\end{align*}
for all $l=1,...,m-1, i=0,...,m-l,$ and $k=1,...,m-l-j$.\\

As before, to extend the argument to include $\a\geq 1$, we must include the next term in the Taylor finite difference
\begin{align*}
I_{\bj,i}[\p^l(\rho u)](x)&=\int_{\mathbb{T}^n}\frac{h(|z|)}{|z|^{n+\alpha-\tau}} \frac{\prod_{k=1}^{m-l-i} \left( \int_{\Omega(x,x+z)} \partial^k \rho(\xi) \dxi\right)^{j_k}}{d^{\tau+(|\bj|+1)n}}\times\\
& \ \ \ \ \ \ \ \ \times\left(\int_{\Omega(x,x+z)}\partial^i \nabla \rho(\xi) \dxi \right)\cdot [\delta_z\p^l(\rho u)(x)-z\nabla \p^l\rho u(x)] \dz\\
&+\int_{\mathbb{T}^n}\frac{h(|z|)}{|z|^{n+\alpha-\tau}} \int_{\Omega(x,x+z)}\partial^i \nabla \rho(\xi) \dxi \frac{\prod_{k=1}^{m-l-i} \left( \int_{\Omega(x,x+z)} \partial^k \rho(\xi) \dxi\right)^{j_k}}{d^{\tau+(|\bj|+1)n}} \cdot z\nabla \p^l (\rho u)(x) \dz\\
&:=I_{\bj,i,1}[\p^l(\rho u)](x)+I_{\bj,i,2}[\p^l(\rho u)](x)
\end{align*}
Again, the estimate on $I_{\bj,i,1}[\p^l(\rho u)]$ goes as before, and for $I_{\bj,i,2}[\p^l(\rho u)](x)$ we symmetrize
\begin{align*}
I_{\bj,i,2}[\p^l(\rho u)](x)&=\int_{\mathbb{T}^n}\frac{h(|z|)}{|z|^{n+\alpha-\tau}}\Big[ \int_{\Omega(x,x+z)}\partial^i \nabla \rho(\xi) \dxi \frac{\prod_{k=1}^{m-l-i} \left( \int_{\Omega(x,x+z)} \partial^k \rho(\xi) \dxi\right)^{j_k}}{d^{\tau+(|\bj|+1)n}(x,x+z)}\\
& \ \ \ \ \ \ \ -\int_{\Omega(x,x-z)}\partial^i \nabla \rho(\xi) \dxi \frac{\prod_{k=1}^{m-l-i} \left( \int_{\Omega(x,x-z)} \partial^k \rho(\xi) \dxi\right)^{j_k}}{d^{\tau+(|\bj|+1)n}(x,x-z)}\Big] \cdot z\nabla \p^l(\rho u(x)) \dz\\
&=\nabla\p^l (\rho u)(x)\int_{\mathbb{T}^n}\frac{h(|z|)}{|z|^{n+\alpha-\tau}} \int_{\Omega(x,x+z)}\partial^i \nabla \rho(\xi) \dxi \prod_{k=1}^{m-l-i} \left( \int_{\Omega(x,x+z)} \partial^k \rho(\xi) \dxi\right)^{j_k}\times\\
& \ \ \ \ \ \ \ \times \left(d^{-\tau-(|\bj|+1)n}(x,x+z)-d^{-\tau-(|\bj|+1)n}(x,x-z)\right)z\dz\\
&+\nabla \p^l(\rho u)(x)\int_{\mathbb{T}^n}\frac{h(|z|)}{|z|^{n+\alpha-\tau}} \int_{\Omega(x,x+z)}\partial^i \nabla \rho(\xi) \dxi d^{-\tau-(|\bj|+1)n}(x,x+z) \times\\
& \ \ \ \ \ \ \ \ \times\left(\prod_{k=1}^{m-l-i} \left( \int_{\Omega(x,x+z)} \partial^k \rho(\xi) \dxi\right)^{j_k}-\prod_{k=1}^{m-l-i} \left( \int_{\Omega(x,x-z)} \partial^k \rho(\xi) \dxi\right)^{j_k}\right) z\dz\\
&+\nabla \p^l(\rho u)(x)\int_{\mathbb{T}^n}\frac{h(|z|)}{|z|^{n+\alpha-\tau}} \prod_{k=1}^{m-l-i} \left( \int_{\Omega(x,x+z)} \partial^k \rho(\xi) \dxi\right)^{j_k} d^{-\tau-(|\bj|+1)n}(x,x+z) \times\\
& \ \ \ \ \ \ \ \ \times\left(\int_{\Omega(x,x+z)}\partial^i \nabla \rho(\xi) \dxi-\int_{\Omega(x,x-z)}\partial^i \nabla \rho(\xi) \dxi \right) z\dz\\
&=I_{\bj,i,2,1}[\p^l(\rho u)](x)+I_{\bj,i,2,3}[\p^l(\rho u)](x)+I_{\bj,i,2,3}[\p^l(\rho u)](x)
\end{align*}
For $I_{\bj,i,2,1}[\p^l(\rho u)](x)$ and $I_{\bj,i,2,2}[\p^l(\rho u)](x)$ we apply the same estimates as above to get
\begin{align*}
|I_{\bj,i,2,1}[\p^l(\rho u)](x)|\leq|\nabla \p^l (\rho u)(x)|\prod_{k=1}^{m-l-i}(M[\p^k\rho](x))^{j_k}M[\p^{i+1}\rho](x)\int_{\mathbb{T}^n}h(|z|)\frac{\dz}{|z|^{n+\a-2}}
\end{align*}
Since $\a<2$, the integral converges, and
\begin{align*}
\|I_{\bj,i,2,1}[\p^l(\rho u)]\|_2^2&\lesssim \|\p^{l+1}(\rho u)\|_{q_2}^2\|\p^{i+1} \rho\|_{q_1}^2\prod_{k=1}^{m-l-i}\|\p^k\rho\|_{p_k}^{2j_k}\\
&\leq  \|\rho u\|_{\dot{H}^{l+1+n(\frac{1}{2}-\frac{1}{q_2})}}^2\| \rho\|_{\dot{H}^{i+1+n(\frac{1}{2}-\frac{1}{q_1})}}^2\prod_{k=1}^{m-l-i}\|\rho\|_{\dot{H}^{k+n(\frac{1}{2}-\frac{1}{p_k})}}^{2j_k}
\end{align*}
Choosing the Holder conjugates as before blends this into the previous case. For $I_{\bj,i,2,2}[\p^l(\rho u)](x)$ we have,
\begin{align*}
|I_{\bj,i,2,2}[\p^l(\rho u)](x)|&\leq |\nabla \p^l (\rho u)(x)|M[\p^{i+1}\rho](x)\times\\
& \times\sum_{k=1}^{m-l-i}\prod_{\substack{\lambda=1\\ i\neq k}}^{m-l-i} (M[\p^{\lambda}\rho](x))^{j_{\lambda}}(M[\p^k\rho](x))^{j_k-1}(D_{s,p_k}\p^k\rho(x))^{1/p_k}\int_{\mathbb{T}^n}\frac{h(|z|)}{|z|^{n+\a-1-s}}\dz
\end{align*}
since $\a<1+s<2$, the integral converges, so for any $k=1,...,m-l-i$,
\begin{align*}
&\|I_{\bj,i,2,2}[\p^l(\rho u)]\|_2^2\leq \|\p^{l+1}(\rho u)\|_{q_2}^2\|\p^{i+1}\rho\|_{q_1}^2\prod_{\substack{\lambda=1\\ i\neq k}}^{m-l-i}\|\p^{\lambda}\rho\|_{p_{\lambda}}^{2j_{\lambda}}\|\p^k\rho\|_{p_k}^{2(j_k-1)}\|\p^k\rho\|_{W^{s,p_k}}^2\\
& \ \ \ \ \ \ \ \leq \|(\rho u)\|_{\dot{H}^{l+1+n(\frac{1}{2}-\frac{1}{q_2})}}^2\|\rho\|_{\dot{H}^{i+1+n(\frac{1}{2}-\frac{1}{q_1})}}^2\prod_{\substack{\lambda=1\\ i\neq k}}^{m-l-i}\|\rho\|_{\dot{H}^{\lambda+n(\frac{1}{2}-\frac{1}{p_{\lambda}})}}^{2j_{\lambda}}\|\rho\|_{\dot{H}^{k+n(\frac{1}{2}-\frac{1}{p_k})}}^{2(j_k-1)}\|\rho\|_{\dot{H}^{k+s+n(\frac{1}{2}-\frac{1}{p_k})}}^2
\end{align*}
again choosing the same Holder conjugates as before gives the necessary bound. Now for $I_{\bj,i,2,3}[\p^l(\rho u)]$ we get,
\begin{align*}
\|I_{\bj,i,2,3}[\p^l(\rho u)]\|_2^2&\lesssim\|(\rho u)\|_{\dot{H}^{l+1+n(\frac{1}{2}-\frac{1}{q_2})}}^2\|\rho\|_{\dot{H}^{i+1+s+n(\frac{1}{2}-\frac{1}{q_1})}}^2\prod_{k=1}^{m-l-i}\|\rho\|_{\dot{H}^{k+n(\frac{1}{2}-\frac{1}{p_k})}}^{2j_k}
\end{align*}
Choosing the same Holder conjugates again gives the desired bound.
Therefore we have the necessary bounds for every term in $\p^m \mathcal{L}_{\nabla \phi \cdot}(\rho u)$.\\
 
Now let us examine $\mathcal{L}_{\phi_t}(\rho)$. Notice that any term in $\p^m\mathcal{L}_{\phi_t}(\rho)$ takes the form
\begin{align*}
I=\int_{\mathbb{T}^n}\frac{h(|z|)}{|z|^{n+\a-\t}}\int_{\O(x,x+z)}\p^i\nabla \cdot(\rho u)(\xi) \dxi\frac{\prod_{k=1}^{m-l-i}\left(\int_{\O(x,x+z)}\p^k\rho(\xi)\dxi\right)^{j_k}}{d^{\t+(|\bj|+1)n}(x,x+z)}\d_z\p^l\rho(x)\dz
\end{align*}
The cases where $l=1,...,m-1$ are estimated exactly the same as the Intermediary case for $\mathcal{L}_{\nabla \phi \cdot}(\rho u)$ above by switching the roles of $\rho u$ and $\rho$ in the increment $\d_z$ and in the first integral that contains the gradient.

Similarly the case where $l=0$ and $i=0,...,m-1$ is taken care of by End Case 5. Further, we have already used the case where $l=m$ during the estimates in End Case 2, and part of the term $l=0, i=m$ in End Case 1. Since $\nabla \p^m(\rho u)=\nabla(u\p^m\rho)+\nabla\p^{m-1}(\rho \p u)$ we still have to estimate the term
\begin{align*}
J=\int_{\mathbb{T}^n}\frac{h(|z|)}{|z|^{n+\a-\t}}\frac{\int_{\O(x,x+z)}\nabla \p^{m-1}(\rho \p u)\dxi}{d^{\t+n}(x,x+z)}\d_z\rho(x)\dz
\end{align*}
For $\a<1$ we use $|\d_z\rho(x)|\leq |\nabla \rho|_{\infty}|z|$ and the maximal function to get
\begin{align*}
|J|&\leq |\nabla \rho|_{\infty}M[\p^m(\rho \p u)](x)\int_{\mathbb{T}^n}\frac{h(|z|)}{|z|^{n+\a-1}}\dz\\
\|J\|_2^2&\leq |\nabla \rho|_{\infty}^2\|\rho\|_{\dot{H}^m}^2\|u\|_{\dH^{m+1}}^2
\end{align*}
For $1\leq \a<2$ we utilize the next Taylor term again and estimate the second of these by symmetrizing and splitting into two parts to get,
\begin{align*}
J_2&=\nabla\rho(x)\int_{\mathbb{T}^n}\frac{h(|z|)}{|z|^{n+\a-\t}}\frac{\int_{\O(x,x+z)}\nabla \p^{m-1}(\rho \p u)\dxi}{d^{\t+n}(x,x+z)}z\dz\\
J_{21}&=\nabla\rho(x)\int_{\mathbb{T}^n}\frac{h(|z|)}{|z|^{n+\a-\t}}\frac{\int_{\O(x,x+z)}\nabla \p^{m-1}(\rho \p u)\dxi-\int_{\O(x,x-z)}\nabla \p^{m-1}(\rho \p u)\dxi}{d^{\t+n}(x,x+z)}z\dz\\
J_{22}&=\nabla\rho(x)\int_{\mathbb{T}^n}\frac{h(|z|)}{|z|^{n+\a-\t}}\int_{\O(x,x-z)}\nabla \p^{m-1}(\rho \p u)\dxi\left(d^{-\t-n}(x,x+z)-d^{-\t-n}(x,x-z)\right)z\dz
\end{align*}
Estimating $J_{21}$ gives
\begin{align*}
|J_{21}|&\leq |\nabla \rho|_{\infty}(D_{s,2}(\p^m(\rho \p u)))^{1/2}\int_{\mathbb{T}^n}\frac{h(|z|)}{|z|^{n+\a-1-s}}\dz\\
\|J_{21}\|_2^2&\leq |\nabla \rho|_{\infty}^2\|\rho\|_{\dH^{m+\a}}^2\|u\|_{\dH^{m+1+s}}^2\\
&\leq Y_m^N\|u\|_{\dH^{m+1}}^2+\e\|u\|_{\dH^{m+1+\frac{\a}{2}}}
\end{align*}
where we used Interpolation and Young's inequality to get the last inequality. Since $1\leq \a<2$ it is possible to find an $s$ such that $s\leq \a/2<1$ for interpolation and $1+s>\a$ to make the above integral finite.

For $J_{22}$ we use the differences in $d^{-\t-n}$ to get 
\begin{align*}
|J_{22}|&\leq |\nabla \rho|_{\infty}^2 M[\p^m(\rho \p u)]\int_{\mathbb{T}^n}\frac{h(|z|)}{|z|^{n+\a-2}}\dz\\
\|J_{22}\|_2^2&\leq |\nabla\rho|_{\infty}^2\|\rho\|_{\dot{H}^m}^2\|u\|_{\dH^{m+1}}^2
\end{align*}

This covers all the terms in $\p^m\mathcal{L}_{\phi_t}(\rho)$. Recalling that the goal is to bound everything by the grand quantity $Y_m^N$, we have shown that 
\begin{align*}
\left\|\p^m(\mathcal{L}_{\phi_t}(\rho)+\mathcal{L}_{\nabla \phi \cdot}(\rho u))\p^me\right\|_2\leq Y_m^N.
\end{align*}
Combined with the transport terms we have estimated in the beginning we therefore have proved the desired a priori bound
\begin{align*}
\ddt\|e\|_{\dH^m}^2\leq CY_m^N.
\end{align*}

\section{Viscous regularization and local existence}\label{s:visc}

To actually produce local solutions we consider viscous regularization of the system
\begin{equation}\label{e:visc}
\begin{split}
\rho_t + \n \cdot (u \rho) & = \e \D \rho \\
u_t + u \cdot \n u & = \cC_\phi(u,\rho) + \e \D u, 
\end{split}
\end{equation}
First, we show that this regularization is sufficient to obtain local solutions via the standard fixed point argument. Second, we show that such regularization does not interfere with the a priori estimates we have obtained in the previous sections.  

To prove local estimates of smooth solutions to \eqref{e:visc} we consider the mild formulation
\begin{equation}
	\begin{split}
	\rho(t) &= e^{\e t \D} \rho_0 - \int_0^t e^{\e(t-s) \D} \n \cdot (u \rho)(s) \ds \\
	u(t) &= e^{\e t \D} u_0 - \int_0^t e^{\e(t-s) \D} u \cdot \n u (s) \ds + \int_0^t e^{\e(t-s) \D}  \cC_\phi(u,\rho)(s) \ds.
	\end{split}
\end{equation}
Let us denote by $Z = (\rho,u)$ the state variable of our system and by $T[Z](t)$ the right hand side of the mild formulation. In order to apply the stadard fixed point argument we have to show that $T$ leaves the set $C([0,T_{\d,\e}); B_\d(Z_0))$ invariant, where $B_\d(Z_0)$ is the ball of radius $\e$ around initial condition $Z_0$, and that it is a contraction. We limit ourselves to showing details for invariance as the estimates involved there are identical to those required to also prove Lipschitzness. 

First we assume that $\rho$ has no vacuum: $\rho_0(x) \geq  c_0>0$. The metric we are using  the same as before $\rho\in \dH^{m+\a} \cap L^1$, $u\in H^{m+1}$.  Note that if $\d>0$ is small enough then for any $\| \rho - \rho_0\|_{\dH^{m+\a}} < \d$ which has the same mass $\int \rho = \int \rho_0$, one obtains 
\begin{equation}
	\rho(x) > \frac12 c_0.
\end{equation}
So, let us assume that $Z \in C([0,T_\d); B_\d(Z_0))$. It is clear that $\|e^{\e t\D} Z_0 - Z_0\|< \frac{\d}{2}$ provided time $t$ is short enough. The $Z$ has some bound $\|Z\| \leq C$.  Using that let us estimate the norms under the integrals. First, recall that $ \| \L_\a e^{\e t\D}\|_{L^2 \to L^2} \lesssim \frac{1}{t^{\a/2}}$. In the case $\a \geq 1$, we have
\begin{multline*}
\left\| \p^m \L_\a \int_0^t e^{\e(t-s) \D} \n \cdot (u \rho)(s) \ds \right\|_2 \leq \int_0^t \frac{1}{(t-s)^{\a/2}} \| \p^{m+1} (u\rho)(s) \|_2 \ds \\
\leq \int_0^t \frac{1}{(t-s)^{\a/2}}  \|u\|_{\dH^{m+1}} \|\rho\|_{\dH^{m+\a}} \ds \leq C^2 t^{1-\a/2} < \frac{\d}{2},
\end{multline*}
provided $T_\d$ is small enough. In the case $\a<1$, we combine instead one full derivatives with the heat semigroup, and the rest $\p^{m+\a}$ gets applied to $u\rho$, which produces a similar bound. 

Moving on to the $u$-equation, we have
\begin{multline*}
	\left\| \p^{m+1} \int_0^t e^{\e(t-s) \D}  u \cdot \n u (s)  \ds \right\|_2 \leq \int_0^t \frac{1}{(t-s)^{1/2}} \| \p^{m} (u \cdot \n u)(s) \|_2 \ds \\
	\leq \int_0^t \frac{1}{(t-s)^{\a/2}}  \|u\|_{\dH^{m+1}} \|u\|_{\dH^{m}}\ds \leq C^2 t^{1/2} < \frac{\d}{2}.
\end{multline*}
As to the commutator form, for $\a\leq 1$ the computation is very similar: we combine one derivative with the heat semigroup and for the rest we use \eqref{e:Lcommmain3}:
\[
\| \p^m \cC_\phi(u,\rho)\|_2 \leq  \|u\|_{m+\a}^N \|\rho\|_{m+\a}^N  < C^{2N},
\]
and the rest follows as before.  When $\a > 1$ we need to use the refined estimate \eqref{e:Lcommmain4}. Namely, it follows from the first in \eqref{e:Lcommmain4} by keeping the highest norms only,
\[
\begin{split}
\|\mathcal{L}_{\phi}  f \|_{\dot{H}^m} &  \lesssim  c_\e\|\rho\|_{\dot{H}^{m-1+\a+\e}}^N\|f\|_{\dot{H}^{m+\alpha}} \\
\|\mathcal{L}_{\phi}  f \|_{\dot{H}^{m-1}} &  \lesssim  c_\e\|\rho\|_{\dot{H}^{m-2+\a+\e}}^N\|f\|_{\dot{H}^{m-1+\alpha}} 
\end{split}
\]
Therefore, by interpolation, we have an estimate in the fractional space $\dH^{m-1+s}$ for $0<s<1$:
\begin{equation}
\|\mathcal{L}_{\phi}  f \|_{\dot{H}^{m-1+s}}  \lesssim  \|\rho\|_{\dot{H}^{m-1+\a+\e}}^N\|f\|_{\dot{H}^{m-1+\alpha + s}} 
\end{equation}
Taking $s = 2-\a$ yields
\begin{equation}
\|\mathcal{L}_{\phi}  f \|_{\dot{H}^{m+1-\a}}  \lesssim  \|\rho\|_{\dot{H}^{m-1+\a+\e}}^N\|f\|_{\dot{H}^{m+1}}. 
\end{equation}
Combining $\a$ derivatives with the heat, and using the inequality above with $\e = 1$, we obtain
\begin{multline*}
\left\|  \int_0^t \L^{\a} e^{\e(t-s) \D} \L^{m+1-\a} \cC_\phi(u,\rho) u (s)  \ds \right\|_2 \leq \int_0^t \frac{1}{(t-s)^{\a/2}} [\| \cL_\phi(u \rho)\|_{\dot{H}^{m+1-\a}}  +  \| u \cL_\phi(\rho)\|_{\dot{H}^{m+1-\a}} ] \ds \\
\leq \int_0^t \frac{1}{(t-s)^{\a/2}}  \|\rho\|_{\dot{H}^{m+\a}}^N \|u\|_{\dH^{m+1}}  \ds \leq C^2 t^{1-\a/2} < \frac{\d}{2}.
\end{multline*}
We have proved that $\|T[Z](t) - Z_0 \| <\d$, and the proof is complete. 

The obtained interval of existence of course depends on $\e$ as it enters into all the estimates of the integrals. In order to conclude the local existence argument we still have to show that our a priori bound
\begin{equation}\label{e:YN}
\ddt Y_m \lesssim Y_m^N
\end{equation}
is independent of $\e$. This would allow us to extend $T_{\e,\d}$ to a time dependent on the initial condition only. Then the classical compactness argument would apply to pass to the limit as $\e \to 0$ in the same state space $C([0,T); (\dH^{m+\a} \cap L^1) \times H^{m+1})$.

It is clear that the $u$-equation will not see the effect of viscous regularization because the term produced by the energy method is $-\e \| \p^{m+2} u \|_2^2$. The $e$-equation, however, will produce several extra terms:
\begin{equation}\label{e:evisc}
e_t+\nabla \cdot (ue)=(\nabla \cdot u)^2-\mathrm{Tr}(\nabla u)^2+\mathcal{L}_{\phi_t}(\rho) +\mathcal{L}_{\nabla \phi \cdot}(\rho u) - 2 \e \cL_{\n \phi}\n \rho- \e \cL_{\D \phi} \rho + \e \D e.
\end{equation}
After the test, the extra terms become
\begin{multline}
	- \e \|e\|_{\dH^{m+1}}^2 -2  \e \lan \p^{m-1} \cL_{\n \phi}\n \rho, \p^{m+1} e \ran - \e \lan \p^{m-1} \cL_{\D \phi} \rho , \p^{m+1} e \ran \\
	\leq - \frac12 \e \|e\|_{\dH^{m+1}}^2 + 8 \e \|\p^{m-1} \cL_{\n \phi}\n \rho\|_2^2 + 4 \e \|\p^{m-1} \cL_{\D \phi} \rho \|_2^2.
\end{multline}
Let us observe that  the residual two terms present special parts of the expansion of the commutator we have estimated in \lem{l:maincomm} for $m \to m+1$. So, from \eqref{e:maincomm} we obtain
\begin{multline*}
\|\p^{m-1} \cL_{\n \phi}\n \rho\|_2^2 + \|\p^{m-1} \cL_{\D \phi} \rho \|_2^2  \lesssim  \|  \rho \|_{\dot{H}^{m+\a}}^N ( \|  \rho \|_{\dot{H}^{m + \frac12 +\a}}^2+ \|  \rho \|_{\dot{H}^{m +1+ \frac{\a}{2}}}^2)+\\
+ ( \|  \rho \|_{\dot{H}^{m+1}}^2+   \|\rho\|_{\dot{H}^{m+\frac12+\a}}^2) \|  \rho \|_{\dot{H}^{2+ \frac{n}{2}}}^2 . 
\end{multline*}
Let us recall that we have another $\e$-gain term from viscous regularization:
\[
-\e \| \p^{m+2} u \|_{\dH^{m+1}}^2  -  \frac12 \e \|e\|_{\dH^{m+1}}^2  \lesssim - \e \rmin^{-2\t/n} \|  \rho \|_{\dot{H}^{m+1+\a}}^2  + \e Y_m^N.
\]
So, by a computation similar to \eqref{e:between} the residual term can be estimated by
\[
\e \|\p^{m-1} \cL_{\n \phi}\n \rho\|_2^2 + \e \|\p^{m-1} \cL_{\D \phi} \rho \|_2^2  \lesssim  \frac12 \e \rmin^{-2\t/n} \|  \rho \|_{\dot{H}^{m+1+\a}}^2  + \e Y_m^N.
\]
So, the total influence of the viscous term on a priori estimates will be an additional $\e Y_m^N$ added to \eqref{e:YN} which has no effect.

Having obtained uniformly bounded solutions $(u^\e,\rho^\e) \in C([0,T); H^{m+1} \times H^{m+\a})$ on a common time interval we pass to the $w^*$-limit in the top space and strong limit in any lower regularity space  $H^{m+1-\d} \times H^{m+\a-\d}$, which guarantees that the limit will actually be weakly continuous in the top space. 

This concludes the proof of local existence.

\section{Appendix: variants of intrinsic definitions of a Sobolev space.}

At another place we considered the quantity
\[
D_s(g) = \int_{\mathbb{T}^{2n}} \frac{h(z)  \left|g \left(x+|z| U_z \th   \right)- g (x) \right|^2 }{|z|^{n+s} }  \dz \dx ,
\]
where $\th \in \p \Omega(0,\be_1)$. It is not essential where exactly $\th$ is localted as long as it is uniformly bounded. 
\begin{lemma}\label{l:SobU}
\[
D_s(g) \leq \| g\|_{\dH^{s/2}}.
\]
\end{lemma} 
\begin{proof}
	\[
\int_{\mathbb{T}^{2n}} \frac{h(z)  \left|g\left(x+ |z| U_z \th   \right)- g (x) \right|^2 }{|z|^{n+s} }  \dz \dx  = \sum_{\bk \in \Z^n} |\widehat{g}(\bk)|^2 \int_{\mathbb{T}^{n}} \frac{h(z)  \left| e^{i \bk \cdot  |z| U_z \th}- 1 \right|^2 }{|z|^{n+s} }  \dz.
\]
Since
\[
\left| e^{i \bk \cdot |z| U_z \th}- 1 \right| \leq \min\{2, |\bk| |z|\}
\]
the splitting of the integral into small scale $|z| <1/|\bk|$ and large scale $|z| > 1/|\bk|$  as in the classical case, shows that the integral is bounded by $|\bk|^s$ which implies the claim. 
\end{proof}
Similar goes the proof of the next lemma
\begin{lemma}\label{l:Sob2} For any $0<\a<2$,
\[
\left\| \int_{\T^n}  [g(\cdot+|z|U_z \th)+g(\cdot-|z|U_z \th) - 2g(\cdot)]    \frac{h(|z|) z_i U_z^{j k }}{|z|^{n+\alpha+1}}  dz  \right\|_2  \leq \| g\|_{\dH^\a}.
\]
\end{lemma}
\begin{proof}
\[
\begin{split}
& \left\| \int_{\T^n}  [g(\cdot+|z|U_z \th)+g(\cdot-|z|U_z \th) - 2g(\cdot)]    \frac{h(|z|) z_i U_z^{j k }}{|z|^{n+\alpha+1}}  dz  \right\|_2^2 =\\
&= \sum_{\bk \in \Z^n} |\widehat{g}(\bk)|^2    \left|\int_{\mathbb{T}^{n}}  (e^{i \bk \cdot  |z| U_z \th}+e^{-i \bk \cdot  |z| U_z \th}- 2)  \frac{h(|z|) z_i U_z^{j k }}{|z|^{n+\alpha+1}}  dz \right|^2\\
& \leq  \sum_{\bk \in \Z^n} |\widehat{g}(\bk)|^2 \left|  \int_{\mathbb{T}^{n}}  |e^{i \bk \cdot  |z| U_z \th}+e^{-i \bk \cdot  |z| U_z \th}- 2|  \frac{h(|z|)}{|z|^{n+\alpha}}  dz \right|^2.
\end{split}	
\]
The integral is estimated with the use of 
\[
|e^{i \bk \cdot  |z| U_z \th}+e^{-i \bk \cdot  |z| U_z \th}- 2|  \leq  \min\{3,|z|^2|\bk|^2 \}
\]
and splitting as before into $|z| <1/|\bk|$ and  $|z| > 1/|\bk|$. The result is $|\bk|^\a$ and the formula follows.
\end{proof}


\begin{thebibliography}{10}
	
	\bibitem{Bal2008}
	M.~Ballerini, N.~Cabibbo, R.~Candelier, A.~Cavagna, E.~Cisbani, I.~Giardina,
	V.~Lecomte, A.~Orlandi, G.~Parisi, A.~Procaccini, M.~Viale, and
	V.~Zdravkovic.
	\newblock Interaction ruling animal collective behavior depends on topological
	rather than metric distance: evidence from a field study.
	\newblock {\em Proc. Natl Acad. Sci. USA}, 105:1232--1237, 2008.
	
	\bibitem{BD2016}
	A.~Blanchet and P.~Degond.
	\newblock Topological interactions in a boltzmann-type framework.
	\newblock {\em J. Stat. Phys.}, 163:41--60, 2016.
	
	\bibitem{BD2017}
	A.~Blanchet and P.~Degond.
	\newblock Kinetic models for topological nearest-neighbor interactions.
	\newblock {\em ArXiv1703.05131v2}, 2017.
	
	\bibitem{CCTT2016}
	Jos\'e~A. Carrillo, Choi, Young-Pil, Eitan Tadmor, and Changhui Tan.
	\newblock Critical thresholds in 1{D} {E}uler equations with non-local forces.
	\newblock {\em Math. Models Methods Appl. Sci.}, 26(1):185--206, 2016.
	
	\bibitem{CCMP2017}
	Jos\'e~A. Carrillo, Young-Pil Choi, Piotr~B. Mucha, and Jan Peszek.
	\newblock Sharp conditions to avoid collisions in singular {C}ucker-{S}male
	interactions.
	\newblock {\em Nonlinear Anal. Real World Appl.}, 37:317--328, 2017.
	
	\bibitem{CS2007a}
	Felipe Cucker and Steve Smale.
	\newblock Emergent behavior in flocks.
	\newblock {\em IEEE Trans. Automat. Control}, 52(5):852--862, 2007.
	
	\bibitem{CS2007b}
	Felipe Cucker and Steve Smale.
	\newblock On the mathematics of emergence.
	\newblock {\em Jpn. J. Math.}, 2(1):197--227, 2007.
	
	\bibitem{DMPW2019}
	Rapha\"{e}l Danchin, Piotr~B. Mucha, Jan Peszek, and Bartosz Wr\'{o}blewski.
	\newblock Regular solutions to the fractional {E}uler alignment system in the
	{B}esov spaces framework.
	\newblock {\em Math. Models Methods Appl. Sci.}, 29(1):89--119, 2019.
	
	\bibitem{DKRT2018}
	Tam Do, Alexander Kiselev, Lenya Ryzhik, and Changhui Tan.
	\newblock Global regularity for the fractional euler alignment system.
	\newblock {\em Archive for Rational Mechanics and Analysis}, 228(1):1--37,
	2018.
	
	\bibitem{Ha2013}
	J.~Haskovec.
	\newblock Flocking dynamics and mean-field limit in the cucker–smale type
	model with topological interactions.
	\newblock {\em Phys. D}, 261(15):42--51, 2013.
	
	\bibitem{HeT2017}
	Siming He and Eitan Tadmor.
	\newblock Global regularity of two-dimensional flocking hydrodynamics.
	\newblock {\em Comptes rendus - Mathématique Ser. I}, 355:795–805, 2017.
	
	\bibitem{SS2016}
	Russell~W. Schwab and Luis Silvestre.
	\newblock Regularity for parabolic integro-differential equations with very
	irregular kernels.
	\newblock {\em Anal. PDE}, 9(3):727--772, 2016.
	
	\bibitem{Shv2018}
	R.~Shvydkoy.
	\newblock Global existence and stability of nearly aligned flocks.
	\newblock {\em https://arxiv.org/abs/1802.08926}, 2018.
	
	\bibitem{ST-topo}
	Roman Shvydkoy and Eitan Tadmor.
	\newblock Topological models for emergent dynamics with short-range
	interactions.
	\newblock https://arxiv.org/abs/1806.01371.
	
	\bibitem{ST1}
	Roman Shvydkoy and Eitan Tadmor.
	\newblock Eulerian dynamics with a commutator forcing.
	\newblock {\em Transactions of Mathematics and Its Applications}, 1(1):tnx001,
	2017.
	
	\bibitem{ST2}
	Roman Shvydkoy and Eitan Tadmor.
	\newblock Eulerian dynamics with a commutator forcing {II}: {F}locking.
	\newblock {\em Discrete Contin. Dyn. Syst.}, 37(11):5503--5520, 2017.
	
	\bibitem{ST3}
	Roman Shvydkoy and Eitan Tadmor.
	\newblock Eulerian dynamics with a commutator forcing {III}. {F}ractional
	diffusion of order {$0<\alpha<1$}.
	\newblock {\em Phys. D}, 376/377:131--137, 2018.
	
	\bibitem{S2012}
	Luis Silvestre.
	\newblock H\"older estimates for advection fractional-diffusion equations.
	\newblock {\em Ann. Sc. Norm. Super. Pisa Cl. Sci. (5)}, 11(4):843--855, 2012.
	
	\bibitem{TT2014}
	Eitan Tadmor and Changhui Tan.
	\newblock Critical thresholds in flocking hydrodynamics with non-local
	alignment.
	\newblock {\em Philos. Trans. R. Soc. Lond. Ser. A Math. Phys. Eng. Sci.},
	372(2028):20130401, 22, 2014.
	
	\bibitem{Tan2017}
	Changhui Tan.
	\newblock Finite time blow up for a fluid mechanics model with nonlocal
	velocity.
	\newblock {\em ArXiv}, 2017.
	
	\bibitem{Wa2017}
	W.~Warren.
	\newblock Behavioral dynamics approach to collective crowd behavior.
	\newblock {\em in ICERM workshop on ``Pedestrian Dynamics: Modeling, Validation
		and Calibration''}, 2017.
	
\end{thebibliography}

\end{document}